\documentclass[12pt,reqno,a4paper]{article}

\usepackage[outer=2.5cm,top=2.5cm,bottom=2.5cm,inner=2.0cm]{geometry}

\usepackage{hyperref} 
\usepackage{amsmath}
\usepackage{amsthm}
\usepackage{amssymb}
\usepackage{amsfonts}
\usepackage{dsfont}
\usepackage{stmaryrd}
\usepackage{commath}
\usepackage{mathrsfs}
\usepackage{enumitem}
\usepackage{graphicx}
\usepackage{tikz}

\font\msbm=msbm10

\theoremstyle{plain}
\newtheorem{theorem}{Theorem}
\newtheorem{lemma}[theorem]{Lemma}
\newtheorem{corollary}[theorem]{Corollary}
\newtheorem{proposition}[theorem]{Proposition}

\theoremstyle{definition}

\newtheorem{remark}[theorem]{Remark}
\def\mathbb#1{\hbox{\msbm{#1}}}

\newcommand{\field}[1]{\ensuremath{\mathds{#1}}}

\newcommand{\re}{\field R}\newcommand{\N}{\field N}
\newcommand{\C}{\ensuremath{\mathbb{C}}}

\newcommand{\Z}{\field Z}

\newcommand{\R}{\field R}

\newcommand{\id}{\ensuremath{\mathrm{id}}}
\newcommand{\Id}{\ensuremath{\mathrm{Id}}}

\newcommand{\be}{\begin{equation}}
\newcommand{\ee}{\end{equation}}
\newcommand{\beq}{\begin{eqnarray}}
\newcommand{\beqq}{\begin{eqnarray*}}
\newcommand{\eeq}{\end{eqnarray}}
\newcommand{\eeqq}{\end{eqnarray*}}

\DeclareMathOperator{\vol}{vol}

\newcommand{\ti}{\times}

\newcommand{\mixedell}[2]{\ell_#1^b(\ell_#2^d)}
\newcommand{\card}{\mathrm{card}}

\hypersetup{
   pdfauthor={Sebastian Mayer, Tino Ullrich},
   pdfkeywords={Besov spaces, entropy numbers, Sch\"utt's theorem},
   pdftitle={Entropy numbers of finite dimensional mixed-norm balls and function space  embeddings with small mixed smoothness}
}

\begin{document}

%\title{On a theorem by Edmunds and Netrusov and consequences for Besov embeddings in regimes of small mixed smoothness}
\title{Entropy numbers of finite dimensional mixed-norm balls and function space  embeddings with small mixed smoothness}
\author{Sebastian Mayer$^{a}$ and Tino Ullrich$^{b, }$\footnote{Corresponding author: tino.ullrich@mathematik.tu-chemnitz.de}
 \\\\
$^a$ Fraunhofer SCAI, Schloss Birlinghoven, Sankt Augustin, Germany\\
$^b$ TU Chemnitz, Fakult\"at f\"ur Mathematik, 09107 Chemnitz}
%\author{Sebastian Mayer, Tino Ullrich}
\maketitle

\begin{abstract}
We study the embedding $\id: \ell_p^b(\ell_q^d) \to \ell_r^b(\ell_u^d)$ and prove matching bounds for the entropy numbers $e_k(\id)$ provided that $0<p<r\leq \infty$ and $0<q\leq u\leq \infty$. Based on this finding, we establish optimal dimension-free asymptotic rates for the entropy numbers of embeddings of Besov and Triebel-Lizorkin spaces of \emph{small} dominating mixed smoothness, which gives a complete answer to Open Problem 6.4 in \cite{dung/ullrich/temlyakov:2015:hyperbolic_cross}. Both results rely on a novel covering construction recently found by Edmunds and Netrusov \cite{edmunds/netrusov:2014:schuett}. 
\end{abstract}

\section{Introduction}

Entropy numbers quantify the degree of compactness of a set, i.e., how well the set can be approximated by a finite set. Given a compact set~$K$ in a quasi-Banach space~$Y$, the~$k$-th entropy number~$e_k(K,Y)$ is defined to be the smallest radius~$\varepsilon > 0$ such that~$K$ can be covered with~$2^{k-1}$ copies of the ball~$\varepsilon B_Y$, i.e., 
$$
		e_k(K,Y):= \inf\Big\{\varepsilon >0~:~\exists y_1,...,y_{2^{k-1}}\text{ such that }K \subset \bigcup\limits_{\ell =1}^{2^{k-1}} y_\ell + \varepsilon B_Y \Big\}\quad,\quad k\in \N\,.
$$
The concept of entropy numbers can be easily extended to operators. Given a compact operator~$T: X \to Y$, where~$X$ and~$Y$ are quasi-Banach spaces, the~$k$-th entropy number of the operator~$T$ is defined to be
$$
e_k(T: X \to Y) := e_k(T(B_X), Y).
$$
If the spaces~$X,Y$ are clear from the context, we will abbreviate~$e_k(T:X \to Y)$ by~$e_k(T)$.

Entropy numbers (or the inverse concept of metric entropy) belong to the fundamental concepts of approximation theory. They appear in various approximation problems, e.g., in the estimation of the decay of operator eigenvalues~\cite{carl:1981,edmunds/triebel:1996:function_spaces,koenig:1986:eigenvalue_distributions}, in the estimation of learning rates for machine learning problems~\cite{temlyakov:2011:greedy,williamson/smola/schoelkopf:1999:entropy}, or in bounding~$s$-numbers like approximation, Gelfand, or Kolmogorov numbers from below~\cite{carl:1981,hinrichs/kolleck/vybiral:2016:carl_quasi}. We note that Gelfand numbers find application in the recent field of compressive sensing~\cite{dirksen/ullrich:2017:gelfand,rauhut/foucart:compressive_sensing,hinrichs/kolleck/vybiral:2016:carl_quasi} and Information Based Complexity in general. Entropy numbers are also closely connected to small ball problems in probability theory~\cite{kuelbs:1993:metric_entropy,li:1999:approximation_metric_entropy}. For further applications and basic properties, we refer to the monographs~\cite{carl/stefani:1990,pietsch:1978:operator}, and the recent survey \cite[Chapter 6]{dung/ullrich/temlyakov:2015:hyperbolic_cross}. 

The subject of this paper is to improve estimates for entropy numbers of 
embeddings between function spaces of dominating mixed smoothness
\begin{equation}\label{emb2}
    \Id:S^{r_0}_{p_0,q_0}A(\Omega) \to S^{r_1}_{p_1,q_1}A^{\dag}(\Omega)\,, \quad A, A^{\dag} \in \{B, F\}\,,
\end{equation}
where~$\Omega \subset \R^n$ is a bounded domain,~$0<p_0, p_1, q_0, q_1 \leq \infty$, and~$r_0-r_1>(1/p_0-1/p_1)_+$. The case~$A=B$ stands for the scale of \emph{Besov spaces} of dominating mixed smoothness, while~$A=F$ refers to the scale of \emph{Triebel-Lizorkin spaces}, which includes classical~$L_p$ and Sobolev spaces of mixed smoothness. That is why \eqref{emb2} also includes the classical embeddings
\begin{equation}
    \Id:S^{r}_{p_0,q_0}A(\Omega) \to L_{p_1}(\Omega)\,, \quad A \in \{B, F\}\,,
\end{equation}
if~$r>1/p_0-1/p_1$. Function space embeddings of this type play a crucial role in hyperbolic cross approximation \cite{dung/ullrich/temlyakov:2015:hyperbolic_cross}. Entropy numbers of such embeddings have been the subject of intense study, see \cite{vybiral:2006:diss}, \cite[Chapt.\ 6]{dung/ullrich/temlyakov:2015:hyperbolic_cross} and the recent papers by A.S. Romanyuk \cite{Ro16_1, Ro16_2, Ro19} and V.N. Temlyakov \cite{Te17}. Note that there is a number of deep open problems connected to the case~$p_1 = \infty$, which reach out to probability and discrepancy theory~\cite[2.6, 6.4]{dung/ullrich/temlyakov:2015:hyperbolic_cross}. 

Typically, one observes asymptotic decays of the form 
$$e_m(\Id) \simeq_n m^{-(r_0-r_1)}(\log m)^{(n-1)\eta},$$
where~$\eta >0$. This behavior is also well-known for~$s$-numbers of these embeddings like approximation, Gelfand, or Kolmogorov numbers, see \cite{dung/ullrich/temlyakov:2015:hyperbolic_cross} and the references therein. Although the main rate is the same as in the univariate case, the dimension still appears in the logarithmic term. We show that the logarithmic term completely disappears in regimes of \emph{small smoothness}
$$1/p_0-1/p_1 < r_0-r_1 \leq 1/q_0-1/q_1.$$
That is, we establish sharp purely polynomial asymptotic bounds of the form 
\begin{equation}\label{loglog}
    e_m(\text{Id}) \simeq_n m^{-(r_0-r_1)}\quad,\quad m\in \N\,,
\end{equation}
which depends on the underlying dimension~$n$ only in the constant. This settles several open questions stated in the literature \cite{dung/ullrich/temlyakov:2015:hyperbolic_cross,vybiral:2006:diss}, see Section \ref{sec:besov}, and makes the framework highly relevant for high-dimensional approximation. 

A key ingredient in the proof of~\eqref{loglog} is a counterpart of Sch\"utt's theorem for the entropy numbers of the embedding 
$$
	\id:\ell_p^b(\ell_q^d) \to \ell_r^b(\ell_u^d), 
$$
where~$0<p<r\leq \infty$ and~$0<q\leq u\leq \infty$. We prove matching bounds for all parameter constellations. A particularly relevant case for the purpose of this paper is the situation where~$b\leq d$ and 
$$
		1/p-1/r > 1/q-1/u \geq 0\,.
$$
Here, we have the surprising behavior
     \begin{equation}\label{eq-0}
     e_k(\id) \simeq
     \begin{cases}
	 1&: 1\leq k\leq \log(bd),\\	
     \left(\frac{\log(ed/k)}{k}\right)^{1/q-1/u} &: \log(bd) \leq k \leq d,\\
     \left(\frac{d}{k}\right)^{1/p - 1/r} d^{-(1/q - 1/u)} &: d \leq k \leq bd,\\
	 b^{-(1/p-1/r)}d^{-(1/q-1/u)}2^{-\frac{k-1}{bd}} &: k\geq bd.
     \end{cases}
     \end{equation}
Note that this relation is not a trivial extension of the classical Sch\"utt result \cite{schuett:1984:entropy}, which reads as 
$$
e_k(\id: \ell_p^b \to \ell_r^b) \simeq \left\{
\begin{array}{rcl}1&:&1\leq k \leq \log(b),\\
\Big(\frac{\log(eb/k)}{k}\Big)^{1/p-1/r}&:&\log(b)\leq k\leq b,\\
2^{-\frac{k-1}{b}}b^{-(1/p-1/r)}&:&k\geq b\,,
\end{array}
\right.
$$
for the norm-$1$-embedding~$\id :\ell^b_p \to \ell^b_r$, where~$0<p\leq r\leq \infty$. In fact, using trivial embeddings would give 
an additional~$\log$-term in the third case of \eqref{eq-0}. The absence of this~$\log$-term makes \eqref{eq-0} interesting and useful as we will see below. 

For~$1 \leq k \leq \log(db)$ and~$k \geq bd$, it requires only trivial and standard volumetric arguments to establish matching bounds for the entropy numbers~$e_k(id: \ell_p^b(\ell_q^d) \to \ell_r^b(\ell_u^d))$. The middle range~$\log(bd)~\leq~k~\leq~bd$ is much more involved. In general, it is far from straightforward to generalize the proof ideas from~$d=1$ (Sch\"utt) to~$d>1$. Fortunately, the crucial work has already been done in a recent work by Edmunds and Netrusov~\cite{edmunds/netrusov:2014:schuett}. They prove a general abstract version of Sch\"utt's theorem for operators between vector-valued sequence spaces. It remains for us to turn these general, abstract bounds into explicit estimates for the entropy numbers~$e_k(\id:\ell_p^b(\ell_q^d)\to \ell_r^b(\ell_u^d))$. 
Unfortunately, the paper~\cite{edmunds/netrusov:2014:schuett} is written very concisely, which makes it difficult to follow the arguments at several points. Hence, we decided to provide some additional, explanatory material. We hope that Section~\ref{sec:edne_notes} helps a broader readership to appreciate the powerful ideas in~\cite{edmunds/netrusov:2014:schuett}, in particular, a novel covering construction based on dyadic grids.

\paragraph*{Outline.} The paper is organized as follows. In Section~\ref{sec:prelim}, we recapitulate basics definitions and results including entropy numbers and Sch\"utt's theorem. Afterwards, in Section~\ref{sec:edne_notes}, we discuss the generalization of Sch\"utt's theorem by~\cite{edmunds/netrusov:2014:schuett}. In Section~\ref{sec:matching_bounds}, we show consequences of this result, including matching bounds for the entropy numbers~$e_k(\id:\ell_p^b(\ell_q^d) \to \ell_r^b(\ell_u^d))$. Finally, we improve upper bounds for the entropy numbers of Besov and Triebel-Lizorkin embeddings in regimes of small smoothness in Section~\ref{sec:besov}.
%We organize this section as follows. First, in Section \ref{sec:entropy_matching_bounds}, we present those parameter constellations where we are able to obtain matching bounds. Subsequently, we present a number of covering and packing constructions upon which the findings in Section \ref{sec:entropy_matching_bounds} rely. Some of the results are formulated in greater generality than needed in Section \ref{sec:entropy_matching_bounds}, which might be of independent interest. Section \ref{sec:entropy:p=r} deals with constructions that exploit the known bavior of~$e_k(B_{\ell_q^d}, \ell_u^d)$. Section \ref{sec:entropy:q=u} is devoted to a generalization of the well-known Edmunds-Triebel covering.

\paragraph{Notation. } As usual~$\N$ denotes the natural numbers,~$\N_0:=\N\cup\{0\}$,~$\Z$ denotes the integers, 
$\R$ the real numbers,~$\R_+$ the positive real numbers, and~$\C$ the complex numbers. For
$a\in \R$ we denote~$a_+ := \max\{a,0\}$. We write~$\log$ for the natural logarithm.~$\R^{m\times n}$ denotes the set of all~$m\times n$-matrices with real entries and~$\R^n$ denotes the Euclidean space. 
Vectors are usually denoted with~$x,y\in \R^n$. For~$0<p\leq \infty$ and~$x\in \R^n$, we use the quasi-norm~$\|x\|_p := (\sum_{i=1}^n
|x_i|^p)^{1/p}$ with the usual modification in the case~$p=\infty$. 
If~$X$ is a (quasi-)normed space, then~$B_X$ denotes its unit ball and the (quasi-)norm
of an element~$x$ in~$X$ is denoted by~$\|x\|_X$. If~$X$ is a Banach space, then we denote its dual by~$X^\ast$. 
We will frequently use the quasi-norm constant, i.e., the smallest constant~$\alpha_X$ satisfying
$$
    \|x+y\|_X \leq \alpha_X(\|x\|_X + \|y\|_X), \qquad \text{for all } x,y\in X.
$$
For a given~$0<p\leq 1$ we say that~$\|\cdot\|_X$ is a~$p$-norm if 
$$
    \|x+y\|^p_X \leq \|x\|_X^p + \|y\|_X^p, \qquad \text{for all } x,y\in X.
$$
As is well known, any quasi-normed space can be equipped with an equivalent~$p$-norm (for a certain~$0<p\leq 1$, see ~\cite{aoki:1942,rolewicz:1957}). If~$T:X\to Y$ is a continuous operator we write~$T\in
\mathcal{L}(X,Y)$ and~$\|T\|$ for its operator (quasi-)norm. The notation~$X \hookrightarrow Y$ indicates that the
identity operator~$\Id:X \to Y$ is continuous. For two non-negative sequences~$(a_n)_{n=1}^{\infty},(b_n)_{n=1}^{\infty}\subset \R$  we
write~$a_n \lesssim b_n$ if there exists a constant~$c>0$ such that~$a_n \leq
c\,b_n$ for all~$n$. We will write~$a_n \simeq b_n$ if~$a_n \lesssim b_n$ and
$b_n \lesssim a_n$. If~$\alpha$ is a set of parameters, then we write~$a_n \lesssim_{\alpha} b_n$ if there exists a constant~$c_{\alpha}>0$ depending only on~$\alpha$ such that~$a_n \leq
c_{\alpha}\,b_n$ for all~$n$.\par 
%In addition, we use the notation~$[d]:=\{1,\ldots,d\}$. 
Let~$b,d \in \N$. For $0<p,q \leq \infty$, the~$bd$-dimensional mixed space~$\ell_p^b(\ell_q^d)$ is defined as the space of all matrices~$x \in \R^{b\ti d}$ equipped with the mixed (quasi-)norm
\[ 
 \|x\|_{p,q} := \left(\sum_{i=1}^b \Big(\sum_{j=1}^d |x_{ij}|^q \Big)^{p/q}\right)^{1/p}, \qquad 0<p<\infty, \; 0<q<\infty\,,
\]
with the usual modification that the corresponding sum is replaced by a maximum in the case that either~$p=\infty$ or~$q=\infty$.
We always refer to the~$\ell_p$-space supported on~$[b]:=\{1,\ldots,b\}$ as the \emph{outer} space and to the~$\ell_q$-space supported on~$[d]$ as the \emph{inner} space. For any~$S\subset[b]\ti[d]$ and~$x\in \R^{b\ti d}$ we define~$x_S$ as the matrix~$(x_S)_{ij} = x_{ij}$ for~$(i,j)\in S$,~$(x_S)_{ij} = 0$ for~$(i,j)\in S^c$.

\section{Entropy numbers and Sch\"utt's theorem}
\label{sec:prelim}
Let us recall basic notions and properties concerning entropy numbers. Let $K$ be a subset of a quasi-Banach space $Y$. Given $\varepsilon > 0$, an \emph{$\varepsilon$-covering} is a set of points~$x_1,\dots,x_n \in K$ such that
\[ 
 K \subset \bigcup_{i=1}^n \big( x_i + \varepsilon B_Y \big)\,.
\]
An \emph{$\varepsilon$-packing} is a set of points~$x_1,\dots, x_m \in K$ such that $\|x_i - x_j\|_Y > \varepsilon$ for pairwise different~$i,j$. The \emph{covering number}~$N_\varepsilon(K,Y)$ is the smallest $n$ such that there exists an~$\varepsilon$-covering of $K$, while the \emph{packing number} $M_\varepsilon(K, Y)$ is the largest $m$ such that there exists an $\varepsilon$-packing of $K$. It is easy to see that 
$$
 M_{2\varepsilon}(K,Y) \leq N_\varepsilon(K,Y) \leq M_\varepsilon(K,Y).
$$
The \emph{metric entropy} is defined to be
$$H_\varepsilon(K,Y) = \log_2 N_\varepsilon(K,Y)\,,$$
see Remark~\ref{rem:notions_of_entropy} for the relation of metric entropy to other notions of entropy.

The $k$-th \emph{entropy number} $e_k(K,Y)$ can be redefined as
\[
 e_k(K, Y) := \inf \{ \varepsilon > 0: H_\varepsilon(K,Y) \leq k-1 \}.
\]
It is easy to see that the sequence of entropy numbers is decaying, i.e., $e_1 \geq e_2 \geq \dots \geq 0$. Moreover, the set $K$ is compact in $X$ if and only if $\lim_{k \to \infty} e_k(K,Y) = 0$.

Let~$T$ denote an operator mapping between two quasi-Banach spaces~$X$ and~$Y$. Recall from the introduction that the operator's entropy numbers are given by
$$e_k(T: X \to Y) = e_k(T(B_X),Y), \quad k \in \N.$$
Clearly, we have 
$$
	\|T\| = e_1(T)	\geq e_2(T) \geq e_3(T) \geq \cdots \geq e_k(T) \geq 0\,.
$$
%In particular, if $T=\id$ the identity operator, we simply have
%$$
%e_k(\id: X \to Y) = e_k(B_X,Y).$$
If~$T_1,T_2$ are both operators from~$X$ to~$Y$, and~$Y$ is a~$\vartheta$-normed space, then the entropy numbers of the sum can be estimated as follows
\begin{equation}\label{ek(T1+T2)}
 e_{k_1 + k_2-1}(T_1 + T_2)^\vartheta \leq e_{k_1}(T_1)^\vartheta + e_{k_2}(T_2)^\vartheta.
\end{equation} 
Moreover, if $S \in \mathcal{L}(X,Y)$ and $R \in \mathcal{L}(Y,Z)$ then 
\begin{equation}\label{ek(RS)}
 e_{k_1 + k_2-1}(R\circ S) \lesssim e_{k}(R)e_{k_2}(S)\,.
\end{equation}
In particular, this gives 
\begin{equation}\label{ek(RS)2}
 e_{k}(R\circ S) \leq e_{k}(R)\|S\|\,.
\end{equation}
For further general properties of entropy numbers and basic estimates, we refer the reader to the monographs~\cite{carl/stefani:1990,lorentz/golitschek/makovoz:1996:cs_advanced,pietsch:s-numbers_and_eigenvalues}. For remarks on the history of entropy number research, see~\cite{carl/stefani:1990,williamson/smola/schoelkopf:1999:entropy}.

In the concrete situation where $X=\ell_p^b$ and $Y=\ell_q^b$ for $0<p \leq q \leq \infty$, the entropy numbers of the embedding $\id: \ell_p^b \to \ell_q^d$ are completely understood in terms of their decay in~$k$ and~$b$. This central result is often referred to as \emph{Sch\"utt's theorem}. For its history and references, see Remark~\ref{rem:schuett}. We only state the interesting case $0<p<q \leq \infty$ here.

\begin{theorem}[Sch\"utt's theorem]\label{lem:schuett}
For~$0<p\leq q \leq \infty$ and~$k,b \in \N$, we have
$$
e_k(\id: \ell_p^b \to \ell_q^b) \simeq \left\{
\begin{array}{rcl}1&:&1\leq k \leq \log(b),\\
\Big(\frac{\log(1+b/k)}{k}\Big)^{1/p-1/q}&:&\log(b)\leq k\leq b,\\
2^{-k/b}b^{1/q-1/p}&:&k\geq b.
\end{array}
\right.
$$
The constants in the estimates do neither depend on~$k$ nor on~$b$.
\end{theorem}

\begin{remark}
Note that $e_k(\id: \ell_\infty^b \to \ell_\infty^b) = 1$ as long as $k \leq b$ because $\|x - y\|_\infty = 2$ for different $x,y \in \{-1,1\}^b$. 
\end{remark}

\begin{remark}\label{rem:schuett}
In 1984, Sch\"utt~\cite{schuett:1984:entropy} gave a proof for the general case of symmetric Banach spaces, which implies Theorem~\ref{lem:schuett} if~$1\leq p \leq q \leq \infty$. In the range~$1 \leq k \leq b$, we upper bound was first proved for all~$0<p \leq q \leq \infty$ by Edmunds and Triebel~\cite{edmunds/triebel:1996:function_spaces} in 1996 by covering the unit ball using suitable sparse vectors. Edmunds and Netrusov~\cite[Thm. 2]{edmunds/netrusov:1998:entropy} generalized this covering construction in 1998 to arbitrary quasi-Banach spaces. In the same paper, Edmunds and Netrusov also proved matching lower bounds for general quasi-Banach spaces~\cite[Thm. 2]{edmunds/netrusov:1998:entropy}. K\"uhn~\cite{kuehn:2001:lower_estimate} also proved the lower bound for~$e_k(\id: \ell_p^b \to \ell_q^b)$ with~$0<p\leq q \leq \infty$ in~2001. Both~\cite[Thm. 2]{edmunds/netrusov:1998:entropy} and~\cite{kuehn:2001:lower_estimate} rely on the very same idea to pack the unit ball with sparse vectors and use the fundamental combinatorial fact discussed in Remark~\ref{rem:suboptimal_combinatorics} (ii) below.
In~2000, Gu\'edon and Litvak~\cite[Thm. 6]{guedon/litvak:2000:euclidean_projections} provided an alternative proof of Theorem~\ref{lem:schuett} that relies completely on interpolation arguments and improved the constants in the upper bound.
%Finally, Richter and Stehling~\cite{richter/stehling:2011:entropy} showed in 2011 that~$e_k(\id: \ell_\infty^b \to \ell_\infty^b) = 1$ as long as~$k \leq b$.
\end{remark}

\begin{remark}\label{rem:notions_of_entropy}
The concept of metric entropy for compact sets has been introduced independently by Kolmogorov~\cite{kolmogorov:1956:entropy} and Pontrjagin and Schnirelmann~\cite{pontrjagin/schnirelmann:1932:entropy}. It should not be confused with the metric entropy of a dynamical system, which also has been introduced by Kolmogorov~\cite{kolmogorov:1956:entropy_dynamical_system}. The latter entropy is also called \emph{Kolmogorov-Sinai entropy} or \emph{measure-theoretic entropy}. However, these two notions of metric entropy are related~\cite{akashi:1986:entropy}. There is also a deep connection between Kolmogorov-Sinai entropy and the notions of information entropy and thermodynamic entropy~\cite{billingsley:1965:ergodic_theory}.
\end{remark}
\section{Edmunds-Netrusov revisited}
\label{sec:edne_notes}

In addition to Sch\"utt's theorem, the main tool that we employ in this work is a powerful result by Edmunds and Netrusov~\cite{edmunds/netrusov:2014:schuett}. They prove a generalization of Sch\"utt's theorem for vector-valued sequence spaces. Let us restate the part of their result that is relevant for us.

\begin{theorem}[Theorems 3.1 and 3.2 in \cite{edmunds/netrusov:2014:schuett}]\label{res:edne}
Let $b \in \N$ such that $b \geq 2$, $0<p\leq r\leq \infty$ and let $X$ and $Y$ be~$\gamma$-normed quasi-Banach spaces. For $k,m \in \N$ such that $m \leq k$, let
\[ 
 D(m,k) = \max_{\ell \in \N, m \leq \ell \leq k} (\ell/k)^{1/p-1/r} e_\ell(\id:X \to Y),
\]
and
\[ 
 A(k,b) = \max\left\lbrace \|\id: X \to Y\|\left(\frac{\log(eb/k)}{k}\right)^{1/p-1/r}, D(1,k) \right\rbrace.
\]
For $k \geq \log_2(b)$, we have the following.
\begin{enumerate}[label=(\roman*)]
\item If $k \leq b$, then
\[ 
 e_k(\id:\ell_p^b(X) \to \ell_r^b(Y)) \simeq A(k,b).
\]
\item If $k \geq b$, then there are absolute constants $c_1, c_2$ such that
\[ 
  D(c_1 k/b, k) \lesssim e_k(\id:\ell_p^b(X) \to \ell_r^b(Y)) \lesssim  D(c_2 k/b,k).
\]
\end{enumerate}
\end{theorem}

\noindent Theorem~\ref{res:edne} states abstract lower and upper bounds that are ``matching'' in the sense that both have the same functional form. At first glance, this functional form is not obvious to expect and not easy to interpret. In addition, we found it difficult to follow the arguments in~\cite{edmunds/netrusov:2014:schuett} at several points due to its succinct style of presentation. We thus believe that it is of value to review their key arguments and to provide some additional material that makes Theorem~\ref{res:edne} more comprehensible. This is the subject of the remainder of this section. The reader who is only interested in applications of Theorem~\ref{res:edne} may proceed directly to Section~\ref{sec:matching_bounds}. 

\begin{remark}
Theorems 3.1 and 3.2 in \cite{edmunds/netrusov:2014:schuett} are only stated for $0<p<r\leq \infty$. However, these theorems also hold true for $p=r$. First note that in the latter case, we have
\[ 
 D(m,k) = e_m(\id: X \to Y), \qquad A(k,b) = \|\id: X \to Y\|.
\]
Now for $k \geq b$, Theorem~\ref{res:edne} has been proved in \cite[Thm. 4.3]{mayer:2018:thesis}. For $k \leq b$, the lower bound in Theorem~\ref{res:edne} is a consequence of \cite[Thm. 4.3]{mayer:2018:thesis} in combination with arguments analogous to Remark~\ref{rem:suboptimal_combinatorics}; the upper bound is trivial.
\end{remark}

\subsection{A special case to begin with}\label{sec:mixed:q=u}

If $p=r=\infty$ it is  clear that one simply has to take $b$-fold Cartesian products of the optimal covering and packing of $B_X$ in~$Y$ to obtain the bounds
$$
 \frac{1}{2} \, e_{k+1}(\id:X \to Y)\leq e_{kb}(\id:\ell_\infty^b(X) \to \ell_\infty^b(Y)) \leq e_k(\id:X \to Y), \qquad k \in \N.
$$
In any other case, simple Cartesian products will not be good enough. 

The special case of equal inner spaces $X=Y$ also allows for a rather straightforward solution if the dimension of the inner space is finite. For an easier understanding of the contribution in \cite{edmunds/netrusov:2014:schuett}, see Theorem \ref{res:edne} above, we find it instructive to give a direct proof of this special case and point out its limitations. Indeed, a straightforward generalization of the well-known Edmunds-Triebel covering construction~\cite{edmunds/triebel:1996:function_spaces} based on volume arguments will do the job to establish the optimal upper bound. Recall that the essence of this covering construction is a result from best $s$-term approximation, sometimes referred to as Stechkin's inequality, see \cite[Sect.\ 7.4]{dung/ullrich/temlyakov:2015:hyperbolic_cross}, which yields a $s^{-1/p+1/r}$-covering of $B_{\ell_p^b}$ in $\ell_r^b$ using only $s$-sparse vectors. We simply have to extend this approach to row-sparse matrices. To improve readability, we will omit some technical details in the following proof. 
%The interested reader may find these in~\cite[Appendix~A]{mayer:2018:thesis}, including explicit expressions for constants.

\begin{proposition}\label{res:generalization_Edmunds_Triebel}
Let $0<p \leq r \leq \infty$ and $X$ be $\R^d$ (quasi-)normed with $\|\cdot\|_X$. Let further  $b,d \in \N$ and $d>5$. Then, for $1 \leq k \leq bd$,
\begin{align*}
	 e_k(\id : \ell_p^b(X) \to  \ell_r^b(X)) \lesssim_{p,r}
     \begin{cases}
	 1&: 1\leq k\leq \max\{\log(b),d\},\\	
     \left[\frac{\log(eb/k)+d}{k}\right]^{1/p - 1/r} &:  \max\{\log(b),d\} \leq k \leq bd,\\
     b^{-(1/p-1/r)}2^{-(k-1)/(bd)} &:k\ \geq bd.
     \end{cases}
\end{align*}
\end{proposition}

\begin{proof}
The first case is trivial. The last case follows from volumetric arguments using the recent findings in Section \cite[Sect.\ 3.2]{KeVy17}. By these we know that 
\begin{equation}\label{volume}
	\vol(B_{\ell_p^b(X)})^{1/(bd)} = \frac{\Gamma(1+d/p)^{1/d}}{\Gamma(1+db/p)^{1/(db)}}\cdot \vol(B_X)^{1/d}\,.
\end{equation}
and for $\vol(B_{\ell_r^b(X)})^{1/(bd)}$ accordingly. For $k>bd$ we use the standard volume argument to obtain 
\begin{equation}\label{vol_arg}
	e_k \lesssim \Big[\frac{\vol(B_{\ell_r^d(X)})}{\vol(B_{\ell_p^d(X)})}\Big]^{1/(db)}2^{-(k-1)/bd} \simeq_{p,r} 
	b^{-(1/p-1/r)}2^{-(k-1)/(bd)}\,.
\end{equation}
For the second case let $s \in [b]$. Clearly, we have that
\begin{align}\label{eq:union_of_sparse} 
 B_{\ell_p(X)} \subseteq \bigcup_{I \subseteq [b]: \sharp I = s} B_I,
\end{align}
where $ B_I := \{ x \in B_{\mixedell{p}{q}} : \|x_{i\cdot}\|_X \geq \|x_{k\cdot}\|_X \text{ for } i \in I, k\in [b]\setminus I\}$. When we replace the~$s$ rows with the largest $\|\cdot\|_X$-(quasi-)norm by $0$ in $x \in B_I$, then the resulting matrix has a~$\ell_r^b(X)$-(quasi-)norm of at most $s^{-(1/p-1/r)}$, which follows from a well-known relation for best~$s$-term approximation in $\ell_r$. Hence, if we wish to cover the set~$B_I$ by balls of radius~$\varepsilon \simeq s^{-(1/p-1/r)}$, it suffices to take care of the $s$ largest components of the vectors in~$B_I$. That is, we take a suitable covering of $B_{\ell_p^s(X)}$ in~$\ell_r^s(X)$ and append $b-s$ zero rows to every matrix of the covering. A similar volumetric argument as above in \eqref{volume}, \eqref{vol_arg} tells us that
\begin{align}\label{eq:bottleneck} 
 e_{1+c_{p,q}sd}(\id:\ell_p^s(X) \to \ell_r^s(X) ) \lesssim_{p,r} \Big[\frac{\vol(B_{\ell_r^s(X)})}{\vol(B_{\ell_p^s(X)})}\Big]^{1/(ds)} \lesssim_{p,r} 
 		s^{-(1/p-1/r)}\,,
\end{align}
so that we obtain a covering of $B_I$ with cardinality~$2^{c_{p,q}sd}$.

Combining the coverings for all possible index sets~$I$ yields an $\varepsilon$-covering $U$ of $B_{\ell_p^b(X)}$ in~$B_{\ell_r^b(X)}$, where $\varepsilon \simeq s^{-1/p+1/r}$, with cardinality
\[
 \sharp U = {b \choose s} 2^{c_{p,q}sd}.
\]
Now, given $k \in [bd]$, we choose 
\begin{align*} %\label{sec:mixed:eq:no_technicalities1}
s \simeq \frac{k}{\log(eb/k)+d}
\end{align*}
such that
\begin{align*} %\label{sec:mixed:eq:no_technicalities2}
\log(\sharp U)  \lesssim s (\log(eb/s)+d) \leq k-1
\end{align*}
is assured. Consequently, we obtain the upper bound
\begin{align*} %\label{sec:mixed:eq:no_technicalities3} 
 e_k(\id:\ell_p^b(X) \to \ell_r^b(X)) \lesssim s^{-1/p+1/r} \lesssim \left(\frac{\log(eb/k)+d}{k}\right)^{1/p-1/r}.
\end{align*}
\end{proof}

\begin{remark}\label{remprop5} One way to obtain the matching lower bound in the case $X = Y$ is to generalize the proof idea underlying Sch\"utt's theorem (Theorem~\ref{lem:schuett}) in the case that $\log(b) \leq k \leq b$. However, the standard combinatorial lemma is not sufficient here. A suitable packing to do this generalization has already been considered in~\cite[Prop. 5.3]{dirksen/ullrich:2017:gelfand}. See also Remark \ref{rem:suboptimal_combinatorics} below. 
\end{remark}

%\begin{remark}
%Without any modification, the proof of Theorem \ref{sec:mixed:res:equal_inner_spaces_lb} yields the following lower bounds in case that $q < u$.
%\begin{align*}
% e_k(B_{\mixedell{p}{q}}, \mixedell{r}{u}) \gtrsim
% \begin{cases}
% 1 &: 1 \leq k \leq \log_2(bd),\\
% (1/d)^{1/q-1/u} &: \log_2(bd) \leq k \leq d + \log_2(b),\\
% (1/d)^{1/q-1/u} \left[\frac{\log(eb/k))+d}{k}\right]^{1/p - 1/r} &: d + \log_2(b) \leq k \leq bd.
% \end{cases}
%\end{align*}
%These are not optimal, however; see Theorem \ref{sec:mixed:res:entropy_matching_bounds}.
%\end{remark}

\subsection{The covering construction by Edmunds and Netrusov}\label{sec:edne}

The generalized Edmunds-Triebel covering is optimal for finite dimensional $X=Y$, see Proposition~\ref{res:generalization_Edmunds_Triebel} in the previous section. In the general situation, where $X$ is compactly embedded into $Y$, it seems that the volumetric arguments underlying~\eqref{eq:bottleneck} are too coarse to obtain sharp estimates (at least in the finite dimensional situation). The main contribution of~\cite{edmunds/netrusov:2014:schuett} is a covering construction which resolves this shortcoming by not using volumetric arguments at all. In particular, $X$ and $Y$ do not have to be finite dimensional. We give a detailed recapitulation of their idea in this section. For some comments concerning the lower bound in Theorem~\ref{res:edne}, see Remark~\ref{rem:suboptimal_combinatorics} at the end of this section.

The covering in~\cite{edmunds/netrusov:2014:schuett} works in the very general situation where we are given quasi-Banach spaces $X_1,\dots, X_b$ and $Y_1,\dots, Y_b$, see Proposition \ref{lem:ub_by_dyadic_covering} below. The basic idea is to cover the unit ball $B_{\ell_p(\{X_i\}_{i=1}^b)}$ by $N$ cuboids
\begin{align}\label{eq:cuboids} 
 U(v^i) = v_1^i B_{X_1}  \times \cdots \times v_b^i B_{X_b}, 
\end{align}
where $v^1,\dots,v^N \in \R_+^b$ and~$N$ is exponential in~$b$ (think of each cuboid as an anisotropically rescaled version of $B_{\ell_\infty(\{X_i\}_{i=1}^b)}$). The crux is to find suitable vectors $v^i$ such that an optimal covering can be reached by covering the cuboid~$U(v^i)$ using a product of optimal coverings of $B_{X_1}$,\dots,$B_{X_b}$. Edmunds and Netrusov~\cite{edmunds/netrusov:2014:schuett} had the idea to consider vectors that form a \emph{dyadic grid} derived from the simplex
$$
 S(b) = \Big\lbrace x \in [0,1]^b: \sum_{i=1}^b x_i \leq 1 \Big\rbrace, \qquad b \in \N.
$$
The dyadic grid is constructed with the help of the following mapping.
Let
$$\upsilon_0: \R_+ \to \{2^k: k \in \N_0\}, \; x \mapsto 2^{\max\{0,\lceil \log_2(x) \rceil \}},$$
and for $x \in [0,1]^b$, put
$$\upsilon(x) := b^{-1} (\upsilon_0(bx_1),\dots, \upsilon_0(bx_b)).$$
This mapping $\upsilon$ leads to a finite grid with the following properties.

\begin{lemma}[Simplification of Lemma 2.2 in \cite{edmunds/netrusov:2014:schuett}]\label{chapter:entropy:lem:dyadic_grid}
For $b \in \N$, let $\Gamma(b) = \upsilon(S(b))$. The set~$\Gamma(b)$ has the following properties.
\begin{enumerate}[label=(\roman*)]
\item For all $u \in S(b)$, there is $v \in \Gamma(b)$ such that $u_i \leq v_i$ for all $i \in [b]$.
\item For all $v \in \Gamma(b)$, we have $\|v\|_1 \leq 2$.
\item For all $v \in \Gamma(b)$, we have $bv_i \in \N$ for each $i \in [b]$.
\item We have $\sharp \Gamma(b) \leq 2^{3b}$.
\end{enumerate}
\end{lemma}

\begin{proof}
Given $x \in S(b)$, let $v = \upsilon(x)$. We clearly have $\sum_{i=1}^b v_i \leq 2$ and $b v_i \in \N$ for all indices~$i=1,\dots,b$. Further
\[
 \sharp \{ i \in [b] :  v_i \geq t \} \leq 2/t,
\]
which is a crucial property to estimate the cardinality of the set $\Gamma(b)$.
Let
\begin{align*}
 B(v,k) &:= \{  i \in [b]: v_i = 2^k/b  \}, \quad k \in \N_0,\\
 C(v,k) &:= \{ i \in [b]: v_i \geq 2^k/b \}.
\end{align*}
Clearly, $\sharp B(v,k) \leq \sharp C(v,k) \leq \min \{b, b2^{1-k}\}$. Varying over all elements in the simplex, $B(v,0)$ can be any of the $2^b$ subsets of $[b]$. Fixing $B(v,0)$, there are at most $2^b$ possibilities for $B(v,1)$. Fixing $B(v,0)$ up to $B(v,k-1)$, there are at most $2^{b2^{1-k}}$ possibilities for $B(v,k)$. Hence, in total the set $\Gamma(b)$ may contain at most
\[ 
 2^b \cdot 2^b \cdot \sum_{k=2}^\infty 2^{b2^{k-1}} = 2^{3b}
\]
many elements.
\end{proof}

\noindent The dyadic grid according to Lemma \ref{chapter:entropy:lem:dyadic_grid} allows to establish the following upper bound on entropy numbers. 

\begin{proposition}[Reformulation of Lemma 2.3 in \cite{edmunds/netrusov:2014:schuett}]
\label{lem:ub_by_dyadic_covering} Let $X_1,\dots,X_b$ and $Y_1,\dots,Y_b$ be quasi-Banach spaces, let~$0<p\leq r\leq \infty$, and let $k\in \N$ such that $k \geq 8b$. Then, we have
\begin{equation} 
 \begin{split}
   &e_{k+1}(\id:\ell_p(\{X_i\}_{i=1}^b) \to \ell_r(\{Y_i\}_{i=1}^b)) \\
   &~~~~~~~~\leq 2^{1/r} \, 8^{1/p-1/r} \max_{i \in [b]} \max_{\lfloor\frac{3k}{8b}\rfloor \leq m \leq k} (m/k)^{1/p-1/r} e_{m}(\id:X_i \to Y_i).
 \end{split}
\end{equation}
\end{proposition}
\begin{proof}
Consider the transformed grid
\[ 
 \Gamma(b,p) = \big\{ (v_1^{1/p}, \dots, v_b^{1/p}): v \in \Gamma(b)  \big\}.
\]
By Lemma \ref{chapter:entropy:lem:dyadic_grid} (i), we have
\[ 
 B_{\ell_p(\{X_i\}_{i=1}^b)} \subset \bigcup_{v \in \Gamma(b,p)} U(v),
\]
where $U(v)$ is the cuboid defined in \eqref{eq:cuboids}.

Let~$v \in \Gamma(b,p)$ be given by $v = (v_1^{1/p}, \dots, v_b^{1/p})$. For each
$$
 m_i = \lfloor 1/2(k/b-2) \rfloor b v_i, \quad i \in [b],
$$
let $\mathcal C_i$ be a~$e_{m_i}(v_i^{1/p}B_{X_i},Y_i)$-covering. Then, for every $x \in U(v)$, there is $y \in \ell_r^b(Y)$ such that~$y_{i\cdot} \in \mathcal C_i$ and
\begin{align*} 
 \|x - y\|_{\ell_r^b(Y)} \leq \left( \sum_{i=1}^b v_i^{r/p} e_{m_i}(B_{X_i},Y_i)^r \right)^{1/r} \leq \left( \sum_{i=1}^b v_i \right)^{1/r} \max_{i=1,\dots,b} \max_{j=1,\dots,b} v_j^{1/p-1/r} e_{m_j}(B_{X_i},Y_i).
\end{align*}
By construction of the set $\Gamma(b)$, we have $\left( \sum_{i=1}^b v_i \right)^{1/r} \leq 2^{1/r}$ and
\begin{align*} 
 \max_{j=1,\dots,b} v_j^{1/p-1/r} e_{m_j}(B_{X_i},Y_i)
 &\leq 8^{1/p-1/r} \max_{m=\lfloor \frac{3k}{8b} \rfloor,\dots,k} (m/k)^{1/p-1/r} e_{m}(B_{X_i},Y_i). 
\end{align*}
Finally, note that the product $\mathcal C_1 \times \cdots \times \mathcal C_s$
has cardinality
$$
 \prod_{i=1}^{b} 2^{m_i-1} \leq 2^{k-3b},
$$
which, in combination with $\sharp \Gamma(b,p) \leq 2^{3b}$, implies the desired result.
\end{proof}

Proposition~\ref{lem:ub_by_dyadic_covering} is not the complete final answer. For $k \leq b$, we have to modify the proof of Proposition~\ref{res:generalization_Edmunds_Triebel}. We sketch the proof and refer to the proof of \cite[Thm 3.1]{edmunds/netrusov:2014:schuett} for technical details.

\begin{proposition}\label{prop_small_k} Let $\log_2(b) \leq k \leq b$. Then, we have
\[ 
 e_k(\id:\ell_p^b(X) \to \ell_r^b(Y)) \lesssim A(k,b),
\]
where $A(k,b)$ is defined in Theorem~\ref{res:edne}.
\end{proposition}
\begin{proof}[Proof sketch]
Let $s \in [k]$. It is clear that, analogously to \eqref{eq:union_of_sparse}, we have
\[ 
 B_{\ell_p^b(X)} \subseteq \bigcup_{I \subset [b]:\, \sharp I = s} B_I\,.
\] 
Similar as in Proposition \ref{res:generalization_Edmunds_Triebel}, we can use a covering for $B_{\ell_p^s(X)}$ to construct a covering for $B_I$. Consider now $\varepsilon = e_k(B_{\ell_p^s(X)},\ell_r^s(Y))$ and let $\Gamma_0$ be a minimal $\varepsilon$-covering of $B_{\ell_p^s(X)}$ in $\ell_r^s(Y)$. Let $\Gamma_I = \Gamma_0 \times \{0\}^{[b]\setminus I}$. Then, for every $x \in B_I$, there is $y \in \Gamma_I$ such that
\[ 
 \|x - y\|_{\ell_r^b(Y)} \lesssim_{r,p} \varepsilon + s^{1/r-1/p}\|\id: X \to Y\|, 
\]
where the second term on the right-hand side follows from the best $s$-term approximation result already used in Proposition \ref{res:generalization_Edmunds_Triebel}. Consequently, we have
\begin{align}\label{eq:est1} 
 \|x - y\|_{\ell_r^b(Y)} \lesssim_{r,p} \max \left\lbrace e_k(B_{\ell_p^s(X)},\ell_r^s(Y)),\, s^{1/r-1/p}\|\id: X \to Y\| \right\rbrace.
\end{align} 
In contrast to Proposition \ref{res:generalization_Edmunds_Triebel}, volumetric arguments would now give a suboptimal estimate for the entropy numbers~$e_k(B_{\ell_p^s(X)},\ell_r^s(Y)$. In this general situation, it requires Proposition~\ref{lem:ub_by_dyadic_covering} with~$X_1=\dots=X_b=X$ and $Y_1=\dots=Y_b=Y$ to get the proper estimate. Concretely, since $s \leq k$, we have
\begin{align}\label{eq:est2} 
 e_k(B_{\ell_p^s(X)}, \ell_r^s(Y)) \leq e_k(B_{\ell_p^k(X)}, \ell_r^k(Y))\,,
\end{align}
which leads in combination with Proposition~\ref{lem:ub_by_dyadic_covering} and \eqref{eq:est1} to an upper bound of the form
\[ 
 e_k(B_{\ell_p^b(X)}, \ell_r^b(Y)) \lesssim  \max \left\lbrace s^{1/r-1/p}\|\id: X \to Y\|, \, \max_{m \in [k]} \, (m/k)^{1/p-1/r} e_m(B_X,Y) \right\rbrace.
\]
The usual arguments show that it is optimal to choose $s$ of the order $k/\log(eb/k)$.
\end{proof}

\begin{remark}\label{rem:suboptimal_combinatorics}

We close this section with some remarks concerning the lower bound in Theorem~\ref{res:edne}. Its proof relies on two surprisingly simple observations, see~\cite{edmunds/netrusov:2014:schuett} for details.

{\em (i)} Let $M$ be a maximal $\varepsilon$-packing of $B_X$ in $Y$. Using the Gilbert-Varshamov bound, which is well-known in coding theory~\cite{gilbert:1952,varshamov:1957}, we know that $(2s)^{-1/p}M^{2s} \subset B_{\ell_p^b(X)}$ contains~$N$ elements of mutual distance $s^{1/r-1/p}\varepsilon$, where $N \simeq \card(M)^s$. This leads to the lower bound
 $$e_{ms}(B_{\ell_p^s(X)}, \ell_r^s(Y)) \gtrsim s^{1/r - 1/p} e_{4m+6}(B_X, Y),$$
 see~\cite[p. 68]{mayer:2018:thesis} and~\cite[Lem. 2.6]{edmunds/netrusov:2014:schuett} for a more general formulation. Given $k \in \N$, we have to make a good choice for the dimension $s$ to maximize the lower bound. Choose~$s = k/m$ for some~$m \in [k]$ to obtain
 $$e_k(B_{\ell_p^{s}(X)}, \ell_r^{s}(Y)) \gtrsim (m/k)^{1/p-1/r} e_{4m+6}(B_X,Y).$$
 If $k \leq b$, we conclude
 $$
  e_k(B_{\ell_p^{b}(X)}, \ell_r^{b}(Y)) \gtrsim \max_{m \in [k]} \; (m/k)^{1/p-1/r} e_{4m+6}(B_X,Y).
 $$
 If $k \geq b$, then $m \geq k/b$ guarantees $s=k/m \leq b$ and thus
 $$
   e_k(B_{\ell_p^{b}(X)}, \ell_r^{b}(Y)) \gtrsim \max_{k/b \leq m \leq k} \; (m/k)^{1/p-1/r} e_{4m+6}(B_X,Y).
 $$

{\em (ii)} Choose a vector $x \in B_X$ such that $$\|x\|_Y \geq \frac{1}{2}\|\id:X \to Y\|.$$ We construct a packing by building row-sparse matrices, where the nonzero rows contain copies of $x$ and the row support sets are chosen according to the following combinatorial fact that is well-known in various disciplines of mathematics, see. e.g.,~\cite[Lemma~10.12]{rauhut/foucart:compressive_sensing},~\cite{kuehn:2001:lower_estimate},~\cite{foucart/pajor/rauhut/ullrich:2010:gelfand}  or~\cite[Prop. 2.21, p. 219]{pinkus:nwidth}. 
Given $s,n \in \N$ such that~$0< s < n/2$, there exist subsets~$I_1, \ldots, I_N$ of~$[n]$, where
  \begin{align*}
  N \geq \Big(\frac{n}{8s}\Big)^{s},
  \end{align*}
  such that each subset~$I_i$ has cardinality~$2s$ and
  $$ \sharp(I_i \cap I_j) < s \ \ \text{ whenever } i \neq j.$$
  This leads to the lower bound
  $$e_k(B_{\ell_p^{b}(X)}, \ell_r^{b}(Y)) \gtrsim \|\id: X \to Y\| \left(\frac{\log(eb/k)}{k}\right)^{1/p-1/r}.$$
\noindent In view of the packing construction that we have mentioned in Remark~\ref{remprop5} it is somewhat surprising that it is not necessary to combine the combinatorics of the two observations in order to obtain the optimal abstract bound in Theorem~\ref{res:edne}. An explanation is given in~\cite[Rem. 4.13, p. 69]{mayer:2018:thesis}.

\end{remark}
\section{Consequences of the Edmunds-Netrusov result}\label{sec:matching_bounds}

We discuss some consequences of Theorem~\ref{res:edne}. Let us begin with considering the entropy numbers
\[ 
 e_k := e_k(\id:\ell_p^b(\ell_q^d) \to \ell_r^b(\ell_u^d)), \quad 0<p\leq r \leq \infty, \, 0<q\leq u \leq \infty.
\]
We have the following matching bounds.

\begin{theorem}\label{sec:mixed:res:entropy_matching_bounds}
Let~$0<p \leq r\leq \infty$ and~$0< q \leq u \leq \infty$. Then, we have
\begin{align*}
e_k \simeq
\begin{cases} 
	 1&: 1\leq k\leq \log(bd),\\	
	 b^{-(1/p-1/r)}d^{-(1/q-1/u)}2^{-\frac{k-1}{bd}} &: k\geq bd.
\end{cases}
\end{align*}
For~$\log(bd) \leq k \leq bd$, we have the following case distinctions.
\begin{enumerate}[label=(\roman*)]
  \item Let~$1/p-1/r > 1/q-1/u \geq 0$.
  \begin{enumerate}[label=(i.\alph*)]
   \item In the special case~$q=u$, we have
         \begin{align*}
          e_k \simeq
          \begin{cases}
           1 &: \log(bd) \leq k \leq d,\\
           \left\lbrace \frac{\log(eb/k)+d}{k} \right\rbrace^{1/p-1/r} &: d \leq k \leq bd.
          \end{cases}
         \end{align*}
   \item If~$q<u$ and~$b \leq d$, then
     \begin{align*}
     e_k \simeq
     \begin{cases}
     \left(\frac{\log(ed/k)}{k}\right)^{1/q-1/u} &: \log(bd) \leq k \leq d,\\
     \left(\frac{d}{k}\right)^{1/p - 1/r} d^{1/u - 1/q} &: d \leq k \leq bd.
     \end{cases}
     \end{align*}
   \item If~$q < u$ and~$d \leq b$, then
     \begin{align*}
     e_k \simeq
     \begin{cases}
      \max\left\lbrace \left(\frac{\log(eb/k)}{k}\right)^{1/p-1/r}, \left(\frac{\log(ed/k)}{k}\right)^{1/q-1/u}  \right\rbrace &: \log(bd) \leq k \leq d,\\
      \max \left\lbrace \Big(\frac{\log(eb/k)}{k}\Big)^{1/p-1/r}, \Big(\frac{d}{k}\Big)^{1/p-1/r}d^{1/u-1/q} \right\rbrace &: d \leq k \leq b,\\
      \left(\frac{d}{k}\right)^{1/p-1/r} d^{1/u-1/q} &: b \leq k \leq bd.
     \end{cases}
     \end{align*}
  \end{enumerate}
  \item Let~$1/q-1/u \geq 1/p - 1/r \geq 0$. Then, we have
  \begin{align*} 
   e_k \simeq
      \begin{cases}
      \left(\frac{\log(ebd/k)}{k}\right)^{1/p-1/r} &: \log(bd) \leq k \leq b\log(d),\\
      b^{1/r-1/p} \left(\frac{b\log(ebd/k)}{k}\right)^{1/q-1/u} &: b\log(d) \leq k \leq bd.
      \end{cases}
  \end{align*}
\end{enumerate}
\end{theorem}

\begin{proof} For~$1 \leq k \leq \log(bd)$ and~$k \geq bd$, it requires only standard volumetric arguments, see \cite[Appendix A]{mayer:2018:thesis} for details. Let us also refer to \cite[Lemma 3]{dung:2001:nonlinear}, where this case has been already considered. Let~$D(m,k)$ and~$A(k,b)$ be as defined in Theorem \ref{res:edne}. Moreover, throughout the proof, we write for~$k,l \in \N$,
$$
 s_{k,l} := (l/k)^{1/p-1/r} e_l(\id: \ell_q^d \to \ell_u^d).
$$

\emph{Ad (i.a).}  Since~$q=u$, it follows from Theorem~\ref{lem:schuett} that~$e_l(\id: \ell_q^d \to \ell_u^d) \simeq 1$ for~$1 \leq l \leq d$ and consequently that~$D(1,k) = D(k/b,k) \simeq 1$ and~$A(k,b) \simeq 1$ for all~$k \leq d$. Now, for~$k \geq d$, we have that~$s_{k,l} \simeq (l/k)^{1/p-1/q}$ for~$1 \leq l \leq d$, so the sequence is bounded from above by a monotonically increasing sequence. For~$d \leq l \leq k$, we have
$$
 s_{k,l} \simeq (l/k)^{1/p - 1/r} 2^{-l/d}:=t_{k,l}.
$$
Since~$2^{-l/d}$ decays faster in $l$ than~$(l/k)^{1/p - 1/r}$ increases, we conclude that for~$d \leq l \leq k$, the sequence~$s_{k,l}$ is ``essentially monotonically decreasing''. To be more  precise $t_{k,l}$ attains at~$l=\beta_{p,r}d$ its maximum, where the factor $\beta$ depends only on $p$ and $r$. Hence, the maximum of $s_{k,l}$ can be bounded from above by a constant times the maximum of $t_{k,l}$ and therefore by $c_{p,r}(d/k)^{1/p-1/r}$. Using analogous arguments for~$D(k/b,k)$, we conclude that~$\widetilde D(1,k) = D(k/b,k) \simeq (d/k)^{1/p-1/r}$ and
\[ 
 A(k, b) \simeq \max\left\lbrace \left(\frac{\log(eb/k)}{k}\right)^{1/p-1/r}, \Big(\frac{d}{k}\Big)^{1/p - 1/r} \right\rbrace \simeq \max \left\lbrace \frac{\log(eb/k)}{k}, \frac{d}{k}\right\rbrace^{1/p-1/r}
\]
for~$d \leq k \leq b$.

\emph{Ad (i.b).} Consider now~$0<q<u$ and~$b \leq d$.  For~$\log(bd) \leq k \leq b$, we have in consequence of Theorem~\ref{lem:schuett}, that~$s_{k,l} \simeq (l/k)^{1/p - 1/r}$ for~$1 \leq l \leq \log(d)$ and
$$
 s_{k,l} \simeq (l/k)^{1/p-1/r} \left(\frac{\log(ed/l)}{l}\right)^{1/q-1/u}
$$
Since~$1/p-1/r > 1/q-1/u$, the sequence~$s_{k,l}$ is bounded from above and below up to a constant by a monotonically increasing sequence and consequently, the maximum is attained at~$l=k$ such that~$D(1,k) \simeq (\log(ed/k)/k)^{1/q-1/u}$. Since~$b \leq d$, we further have
$$D(1,k) \leq A(k,b) \lesssim D(1,k).$$
For~$b \leq k \leq d$ we find as before that~$D(k/b, k) \simeq (\log(ed/k)/k)^{1/q-1/u}$ and for~$d < k \leq bd$, we have the estimate~$$D(k/b,k) \simeq \Big(\frac{d}{k}\Big)^{1/p-1/r} d^{1/u-1/q}\,.$$

\emph{Ad (i.c).}
Consider now~$d \le b$. For~$\log(bd) \leq k \leq d$, we find~$D(1,k) \simeq (\log(ed/k)/k)^{1/q-1/u}$ since the sequence~$s_{k,l}$ is bounded from below and above by a sequence that increases monotonically in~$l$. If~$d \leq k \leq b$, then
$$D(1,k) \simeq \Big(\frac{d}{k}\Big)^{1/p-1/r} d^{1/u-1/q}$$
and
\begin{align*}
 A(k,b) &\simeq \max \left\lbrace \Big(\frac{\log(eb/k)}{k}\Big)^{1/p-1/r}, \Big(\frac{d}{k}\Big)^{1/p-1/r} d^{1/u-1/q} \right\rbrace
\end{align*}
Finally, if~$b \leq k \leq bd$, then~$D(k/b,k) \simeq (d/k)^{1/p-1/r} d^{1/u-1/q}$.

\emph{Ad (ii).} For~$\log(bd) \leq k \leq b$, we observe that
\[ 
 D(1,k) \simeq \left(\frac{\log(d)}{k}\right)^{1/p-1/r}
\]
since the term~$e_\ell(B_{\ell_q^d}, \ell_u^d)$ is decaying in~$\ell$ at least as fast as~$(\ell/k)^{1/p-1/r}$ is growing. Hence,
\begin{align*} 
 A(k,b) &\simeq \max \left\lbrace \Big(\frac{\log(eb/k)}{k}\Big)^{1/p-1/r}, \Big(\frac{\log(d)}{k}\Big)^{1/p-1/r}\right\rbrace\\
        &\simeq \left(\frac{\log(ebd/k)}{k}\right)^{1/p-1/r}.
\end{align*}
Next, we consider~$b \leq k \leq b \log(d)$. Since~$k/b \leq \log(d)$, we find
$$D(k/b, k) \simeq (\log(d)/k)^{1/p-1/r} \geq (\log(bd/k)/k)^{1/p-1/r}\,,$$
where we have used~$b/k \leq 1$ in the last estimate. At the same time, since~$k/b \leq \log(d)$, we also have~$\log(bd/k) \gtrsim \log(d)$ and thus
\[ 
 D(k/b, k) \simeq \left( \frac{\log(ebd/k)}{k}\right)^{1/p-1/r}.
\]
Finally, for~$b \log(d) \leq k \leq bd$ it is easy to see that
\[ 
 D(k/b, k) \simeq b^{1/p-1/r} \left(\frac{b\log(ebd/k)}{k}\right)^{1/q-1/u}.
\]
 
\end{proof}

%\begin{remark}
%The bounds in Theorem \ref{res:entropy_matching_bounds} (iii) suggest that the following covering construction is optimal. For small~$k \leq b\log(d)$ we use only matrices in the covering which have at most one non-zero entry per row and optimize the number of non-zero rows. For~$k > b \log(d)$ we do not exploit the block-compressibility anymore and instead rely solely on the stronger compressibility within blocks~$q$ by optimizing the number of non-zero entries per row.
%\end{remark}

\begin{remark}\label{sec:mixed:rem:vybiral}
The upper bound for $k \geq bd$ in Theorem~\ref{sec:mixed:res:entropy_matching_bounds} also follows from \cite[Lem. 3]{dung:2001:nonlinear}. The upper bound in Theorem \ref{sec:mixed:res:entropy_matching_bounds} (ii) has also been proved in \cite[Lem 3.16]{vybiral:2006:diss} for the range~$b \max\{\log(d),\log(b)\} \leq k \leq bd$. The proof there uses the following covering construction, which appeared first in \cite[Proof of Prop. 4]{kuehn/leopold/sickel/skrzypczak:2006:entropy} to our knowledge. Let~$X_1,\dots,X_b$ and~$Y_1,\dots, Y_b$ be (quasi-)Banach spaces and~$0<p,r\leq \infty$. The covering rests on the idea to split the ball~$B_{\ell_p^b({X_1,\dots,X_b})}$ into subsets of matrices with non-increasing rows,
\[ 
B_{\ell_p^b({X_1,\dots,X_b})} \subseteq \bigcup_\pi \{ x \in B_{\ell_p^b({X_1,\dots,X_b})}: \|x_{\pi(1)\cdot}\|_{X_1} \geq \dots \geq \|x_{\pi(b)\cdot}\|_{X_b} \},
\]
where the union is taken over all permutations of~$[b]$. This leads to the upper bound
\begin{align}\label{sec:mixed:eq:covering_KLSS} 
 e_{\sum_{j=1}^b n_j + b \log_2(b)}\big(B_{\ell_p(\{X_j\}_{j=1}^b)}, \ell_r(\{Y_j\}_{j=1}^b)\big) \leq \big( \sum_{j=1}^b j^{-r/p} e_{n_j}(B_{X_j}, Y_j)^r\big)^{1/r}
\end{align}
for~$n_1,\dots,n_b \in \N$. If
\begin{align*}
X=X_1=\dots=X_b=\ell_q^d \quad \text{ and } \quad Y=Y_1=\dots=Y_b=\ell_u^d
\end{align*}
with~$0<q\leq u$, and we chooses~$n_j \simeq j^{-\alpha}$ for some~$0<\alpha<1$ such that
$$\alpha(1/q-1/u) > 1/p - 1/r,$$
then ~\eqref{sec:mixed:eq:covering_KLSS} is strong enough to obtain the upper bound in Theorem~\ref{sec:mixed:res:entropy_matching_bounds} (ii), provided
$$b\max\{\log(d),\log(b)\} \leq k \leq bd.$$
\qed
\end{remark}

Now we increase the level of abstraction and consider mixed norms of higher order. Let, for~$\mu =1,...,b$, the weighted spaces~$X_{\mu}$ and~$Y_{\mu}$ be given by
\begin{eqnarray}\label{lp(lq)}
  		X_{\mu} &=& \ell_p^{b_\mu}(d_{\mu}^{\alpha}\ell_q^{d_\mu})\,,\\
		\nonumber Y_{\mu} &=& \ell_r^{b_\mu}(d_{\mu}^{\beta}\ell_u^{d_\mu})\,,
  \end{eqnarray}
with~$0<p\leq r\leq \infty$,~$0<q\leq u\leq \infty$ and~$\alpha\, \beta \in \re$. The dimensions~$(d_\mu)_\mu$ and~$(b_\mu)_{\mu}$ are non-decreasing natural numbers satisfying~$d_{\mu} \gtrsim b_{\mu}$. These spaces are used as ``inner spaces" in the way that
$$
	X = \ell_p^b((X_\mu)_{\mu = 1,...,b}) \mbox{ and } 
	Y = \ell_r^b((Y_\mu)_{\mu = 1,...,b})\,.
$$
Note that for $x=(x_{\mu,i,j})_{\mu,i,j} \in X$ with $\mu=1,\dots,b$, $i=1,\dots,b_{\mu}$, $j=1,\dots,d_\mu$, the norm is given by
$$
 \|x\|_X = \left\lbrace \sum_{\mu=1}^{b} \sum_{i=1}^{b_\mu} d_\mu^\alpha \left( \sum_{j=1}^{d_\mu} |x_{\mu,i,j}|^q \right)^{p/q} \right\rbrace^{1/p}.
$$
We are interested in the behavior of the entropy numbers 
$$
	e_k(\id:X\to Y) = e_k(\id:\ell_p^b(X_\mu)\to \ell_r^b(Y_\mu))
$$
in the special situation~$1/q-1/u < 1/p-1/r$\,.

\begin{proposition}\label{prop14} Let~$0 \leq 1/q-1/u < 1/p-1/r$ and 
$\alpha-\beta\leq 1/p-1/r-(1/q-1/u)$. Let further~$X, Y$ and~$X_{\mu}$,~$Y_{\mu}$ as above. Then we have for all~$k\geq 8b$ and~$k \geq \max\limits_{\mu = 1,...,b} d_{\mu}$
$$
		e_k(\id:\ell_p^b(X_\mu) \to \ell_r(Y_\mu)) \lesssim 
		\Big(\frac{1}{k}\Big)^{\alpha-\beta+1/q-1/u}\,.
$$

\end{proposition}

\begin{proof} We use Theorem \ref{res:edne}, in particular the upper bound in Proposition \ref{lem:ub_by_dyadic_covering}. Since 	$k\geq 8b$ we obtain 
\begin{eqnarray*}
	e_k(\id:\ell_p^b(X_\mu) \to \ell_r(Y_\mu))  &\lesssim & 
	\max\limits_{\mu = 1,...,b}\max\limits_{1\leq \ell \leq k} \Big(\frac{\ell}{k}\Big)^{1/p-1/r}e_{\ell}(\id:X_{\mu} \to Y_{\mu})\\
	&= & \max\limits_{\mu = 1,...,b}\max\limits_{1\leq \ell \leq k} \Big(\frac{\ell}{k}\Big)^{1/p-1/r}d_\mu^{-(\alpha-\beta)}e_{\ell}(\id:\ell_p^{b_\mu}(\ell_q^{d_{\mu}}) \to \ell_r^{b_\mu}(\ell_u^{d_{\mu}}))\\
	&\lesssim & \max\limits_{\mu = 1,...,b}\Big\{\max\limits_{1\leq \ell\leq d_{\mu}}\Big[\cdots\Big],
	\max\limits_{d_{\mu}\leq \ell\leq k}\Big[\cdots \Big]\Big\}\,.
\end{eqnarray*}
Let us evaluate the first~$\max[\cdots]$. With Theorem \ref{sec:mixed:res:entropy_matching_bounds}, (i.b), (i.c) we have 
\begin{equation}\nonumber
\begin{split}
&\max\limits_{1\leq \ell \leq d_{\mu}} \Big(\frac{\ell}{k}\Big)^{1/p-1/r}d_\mu^{-(\alpha-\beta)}e_{\ell}(\id:\ell_p^{b_\mu}(\ell_q^{d_{\mu}}) \to \ell_r^{b_\mu}(\ell_u^{d_{\mu}}))\\
&\lesssim \Big(\frac{\ell}{k}\Big)^{1/p-1/r}d_\mu^{-(\alpha-\beta)}\Big(\frac{\log(ed_{\mu}/\ell)}{\ell}\Big)^{1/q-1/u}\,.
\end{split}
\end{equation}
Because of~$1/p-1/r > 1/q-1/u$ the maximum is attained for~$\ell = d_{\mu}$, which leads to 
$$
	\max\limits_{1\leq \ell\leq d_{\mu}}\Big[\cdots\Big] \lesssim 
	d_{\mu}^{-(\alpha-\beta)}d_{\mu}^{1/p-1/r}d_{\mu}^{-(1/q-1/u)}k^{-(1/p-1/r)}\,.
$$
Let us discuss the second~$\max[\cdots]$. Using again Proposition~\ref{lem:ub_by_dyadic_covering} we obtain 
$$
	\max\limits_{d_{\mu}\leq \ell \leq k}\Big[\cdots\Big] \lesssim
	k^{-(1/p-1/r)}d_{\mu}^{-(\alpha-\beta)}d_{\mu}^{1/p-1/r}d_{\mu}^{-(1/q-1/u)}\,.
$$
Due to our assumption the exponent for~$d_{\mu}$ is positive in both cases. Since~$k \geq d_\mu$ we may replace~$d_\mu$ by~$k$ to increase the right-hand side. This leads to
$$
		e_k(\id:\ell_p^b(X_\mu) \to \ell_r(Y_\mu)) \lesssim 
		\Big(\frac{1}{k}\Big)^{\alpha-\beta+1/q-1/u}\,.
$$
\end{proof}

We are now aiming for a similar relation for small~$k$. 

\begin{proposition}\label{prop15} Let~$\alpha-\beta>0$ and~$1/p-1/r>1/q-1/u \geq 0$. Then we have for 
$8b \leq k \leq \min\limits_{\mu =1,...,b} d_{\mu}$ the estimate
$$
	e_k(\id:\ell_p^b(X_\mu)\to \ell_r(Y_\mu)) \lesssim \Big(\frac{1}{k}\Big)^{\alpha-\beta+1/q-1/u}\,.	
$$
\end{proposition}

\begin{proof} Again we use Theorem \ref{res:edne}, in particular the upper bound in Proposition \ref{lem:ub_by_dyadic_covering}. This gives
\begin{equation}
  \begin{split}
	e_k &\lesssim \max\limits_{\mu} \max\limits_{1\leq \ell \leq k}
	\Big(\frac{\ell}{k}\Big)^{1/p-1/r}e_{\ell}(\id:X_\mu \to Y_{\mu})\\
	&=\max\limits_{\mu} \max\limits_{1\leq \ell \leq k}
	\Big(\frac{\ell}{k}\Big)^{1/p-1/r}d^{-(\alpha-\beta)}
	e_{\ell}(\id:\ell_p^{b_\mu}(\ell_q^{d_{\mu}})\to\ell_r^{b_\mu}(\ell_u^{d_{\mu}}))\\
	&\lesssim \max\limits_{\mu} \max\limits_{1\leq \ell \leq k}
	\Big(\frac{\ell}{k}\Big)^{1/p-1/r}d_\mu^{-(\alpha-\beta)}
	\Big(\frac{\log(ed_\mu/\ell)}{\ell}\Big)^{1/q-1/u}\,,
  \end{split}
\end{equation}
where we used once again Theorem \ref{sec:mixed:res:entropy_matching_bounds}, (i.b). Clearly, we get
\begin{equation}
 \begin{split}
	e_k &\lesssim \max\limits_{\mu} \max\limits_{1\leq \ell \leq k}
	\Big(\frac{\ell}{k}\Big)^{1/p-1/r}k^{-(\alpha-\beta)}\Big(\frac{k}{d_\mu}\Big)^{\alpha-\beta}\Big(\frac{\log(ed_\mu/\ell)}{\ell}\Big)^{1/q-1/u}\\
	&\lesssim k^{-(\alpha-\beta)}k^{-(1/q-1/u)}
	\Big(\frac{k}{d_\mu}\Big)^{\alpha-\beta}\big(\log(ed_\mu/k)\big)^{1/q-1/u}\,.
\end{split}
\end{equation}
Since the function~$x \mapsto x^{-(\alpha-\beta)}[\log(ex)]^{(1/q-1/u)}$ is bounded on~$[1,\infty)$ we conclude with 
$$
		e_k \lesssim k^{-(\alpha-\beta+1/q-1/u)}\,.
$$
\end{proof}

\section{Polynomial decay of entropy numbers for \\multivariate function space embeddings}\label{sec:besov}
We come to the main subject of this paper, improved upper bounds for entropy numbers of function space embeddings~\eqref{emb2} in regimes of small mixed smoothness.

\subsection{Function spaces of dominating mixed smoothness}
Besov and Triebel-Lizorkin spaces of mixed smoothness are typically defined via a dyadic decomposition on the Fourier side. 
Let $\{\varphi_{j}\}_{j\in \N_0^n}$ be the standard tensorized 
dyadic decomposition of unity, see \cite{schmeisser/triebel:1987:fourier_analysis} and \cite{vybiral:2006:diss}. We further denote by $S'(\R^n)$ the space of tempered distributions and by $D'(\Omega)$ the space of distributions (dual space of $D(\Omega)$, which represents the space of test functions on the bounded domain $\Omega\subset \R^n$). The \emph{Besov space of dominating mixed smoothness}~$S^r_{p,q}B(\R^n)$ with smoothness parameter~$r > 0$ and integrability parameters~$0<p,q\leq \infty$ is given by 
$$
    S^r_{p,q}B(\R^n) := \Big\{f\in S'(\R^n)~:~\|f\|_{S^r_{p,q}B}:=\Big(\sum\limits_{j \in \N_0^n}
    2^{rq\|j\|_1}\|\mathcal F^{-1}[\varphi_{j}\mathcal F f]\|_p^q\Big)^{1/q} < \infty\Big\}\,,
$$
with the usual modification in the case $q=\infty$. The Triebel-Lizorkin space of dominating mixed smoothness $S^r_{p,q}F(\mathbb{R}^n)$ is given by $(p<\infty)$
$$
    S^r_{p,q}F(\R^n) := \Big\{f\in S'(\R^n)~:~\|f\|_{S^r_{p,q}F}:=\Big\|\Big(\sum\limits_{j \in \N_0^n}
    2^{rq\|j\|_1}|\mathcal F^{-1}[\varphi_{j}\mathcal F f](\cdot)|^q\Big)^{1/q}\Big\|_p < \infty\Big\}\,.
$$
The latter scale of spaces contains the classical $L_p$ spaces and Sobolev spaces with dominating mixed smoothness if $1<p<\infty$ and $q=2$, namely we have~$S^0_{p,2}F(\R^n) = L_p(\R^n)$ and~$S^k_{p,2}F(\R^n) = S^k_pW(\R^n)$ for $k\in \N$. Note that we also have~$S^r_{p,p}B(\R^n) = S^r_{p,p}F(\R^n)$ for all~$0<p<\infty$ and $r\in \re$. Though we have the embedding
\begin{equation}\label{emb}
  \begin{split}	
   S^{r_0}_{p_0,q_0}A(\R^n) &\hookrightarrow S^{r_1}_{p_1,q_1}A^{\dag}(\R^n), \quad A,A^{\dag} \in \{B,F\},
  \end{split}	
\end{equation}
for $p_0 \leq p_1$ and $r_0-r_1>1/p_0-1/p_1$, see \cite[Chapt.\ 2]{schmeisser/triebel:1987:fourier_analysis}, the embedding~\eqref{emb} is never compact. Hence, the entropy numbers of embeddings between function spaces defined on the whole $\R^n$ do not converge to zero. We restrict our considerations to spaces on bounded domains $\Omega$. Let $\Omega$ be an arbitrary bounded domain in $\R^n$. Then, we define 
$S^r_{p,q}A(\Omega)$ for $A\in \{B,F\}$ as
$$
    S^r_{p,q}A(\Omega) := \{f\in D'(\Omega)~:~\exists g\in S^r_{p,q}A(\R^n) \text{ such that } g|_{\Omega} = f\}\,
$$
and its (quasi-)norm is given by $\|f\|_{S^r_{p,q}A(\Omega)}:=\inf_{g|_{\Omega} = f} \|g\|_{S^r_{p,q}A}$\,. 
The embedding \eqref{emb} transfers to the bounded domain $\Omega$ and is compact such that the entropy numbers 
decay and converge to zero.

\subsection{Sequence spaces}
The key to establishing the decay rate of entropy numbers for the embedding \eqref{emb2} is a discretization technique which has been developed over the years by several authors beginning with Maiorov \cite{Ma75}. Later, when wavelet isomorphisms have been established, this technique has been refined by Lemarie, Meyer, Triebel and many others. In \cite[Thm.\ 2.10]{vybiral:2006:diss} Vyb{\'i}ral gave the necessary modifications to deal with the above defined $S^r_{p,q}A(\Omega)$ spaces in detail. The main advantage of this approach is to transfer questions for function space embeddings to certain sequence spaces. 

Using sufficiently smooth wavelets with sufficiently many vanishing moments (and the notation from \cite{vybiral:2006:diss}) the mapping
\begin{equation}\label{wave_iso}
    f \mapsto \lambda_{j,m}(f):=\langle f, \psi_{j,m} \rangle\quad,\quad j\in \N_0^n, 
    m \in \Z^n\,,
\end{equation}
represents a sequence spaces isomorphism between $S^r_{p,q}B(\R^n), S^r_{p,q}F(\R^n)$ and
\begin{equation}\nonumber
  \begin{split}
    s^r_{p,q}b &:= \Big\{\lambda=\{\lambda_{j,m}\}_{j, m}~:~
    \|\lambda\|_{s^r_{p,q}b}:=\Big[\sum\limits_{j\in \N_0^n} 2^{(r-1/p)q\|j\|_1}
    \Big(\sum\limits_{m\in \Z^n}|\lambda_{j,m}|^p\Big)^{q/p}\Big]^{1/q}<\infty\Big\},\\
    s^r_{p,q}f &:= \Big\{\lambda=\{\lambda_{j,m}\}_{j, m}~:~
    \|\lambda\|_{s^r_{p,q}f} := \Big\|\Big(\sum_{j \in \N_0^n} 2^{\|j\|_1rq}\Big|\sum_{m \in \Z}\lambda_{j, m}\chi_{j,m}(\cdot)\Big|^q\Big)^{1/q}\Big\|_p<\infty\Big\}\,,
  \end{split}
\end{equation}
respectively, with the usual modification in the case $\max\{p,q\} = \infty$. Here we denote 
for $j \in \N_0^n$ and $m \in \Z^n$
$$Q_{j, m} := \prod\limits_{i=1}^n 2^{-j_i} [m_i-1,m_i+1]$$
and 
$$A_{j}^\Omega = \{m \in \Z^n: Q_{j, m} \cap \Omega \neq \emptyset \}.$$ 
Further, $\chi_{j,m}$ denotes the characteristic function of $Q_{j,m}$.
Consider the sequence spaces
\begin{equation}
\begin{split}
	s_{p,q}^{r}b(\Omega) &:= \{ \lambda = (\lambda_{j, m})_{j \in \N_0^n, m \in A_{j}^\Omega} : \|\lambda\|_{s_{p,q}^{r}			b(\Omega)} < \infty  \}\,,\\
	s_{p,q}^rf(\Omega) &:= \{ \lambda = (\lambda_{j, m})_{j \in \N_0^n, m \in A_{j}^\Omega} : \|\lambda\|_{s_{p,q}^{r}			f(\Omega)} < \infty  \}
\end{split}
\end{equation}
with (quasi-)norms given by
\begin{equation} 
 \begin{split}
 \|\lambda\|_{s_{p,q}^{r}b(\Omega)} &:= \Big[\sum_{j \in \N_0^n} 2^{\|j\|_1(r-1/p)q} \Big( \sum_{m \in A_{j}^\Omega} |\lambda_{j, m}|^p \Big)^{q/p} \Big]^{1/q},\\
 \|\lambda\|_{s_{p,q}^{r}f(\Omega)} &:= \Big\|\Big(\sum_{j \in \N_0^n} 2^{\|j\|_1rq} \Big|\sum_{m \in A_{j}^\Omega}\lambda_{j, m}\chi_{j,m}(\cdot)\Big|^q\Big)^{1/q}\Big\|_p\,.
\end{split}
\end{equation}
Let us also define the following building blocks for $\mu \in \N_0$ fixed 
\begin{equation}\label{block}
  \begin{split}
	s_{p,q}^{r}b(\Omega)_{\mu} &:= \Big\{ \lambda: \|\lambda\|_{s_{p,q}^{r}b(\Omega)_\mu}
	:= \Big[\sum_{\|j\|_1 = \mu} 2^{\|j\|_1(r-1/p)q} \Big( \sum_{m \in A_{j}^\Omega} |\lambda_{j, m}|^p \Big)^{q/p} \Big]^{1/q}<\infty\Big\}\,,\\
s_{p,q}^{r}f(\Omega)_{\mu} &:= \Big\{ \lambda: \|\lambda\|_{s_{p,q}^{r}f(\Omega)_\mu}
	:= \Big\|\Big(\sum_{\|j\|_1 = \mu} 2^{\|j\|_1rq} \Big|\sum_{m \in A_{j}^\Omega}\lambda_{j, m}\chi_{j,m}(\cdot)\Big|^q\Big)^{1/q}\Big\|_p < \infty\Big\}\,.
  \end{split}
\end{equation}
Clearly, for $\mu \in \N_0$ we have $$\sharp \{j \in \N_0^n: \|j\|_1 = \mu \} = {\mu + n -1 \choose \mu} \simeq_d (\mu+1)^{n-1}$$ and $\sharp A_{j}^\Omega \simeq 2^{\|j\|_1} = 2^{\mu}$. Consider 
$$
	\id:s^{r_0}_{p_0,q_0}a(\Omega) \to s^{r_1}_{p_1,q_1}a^{\dag}(\Omega)
$$
for $r_0-r_1 > 1/p_0-1/p_1$ such that the embedding is compact.
Defining the building blocks
\[ 
 (\id_\mu \lambda)_{j, m} := \begin{cases}
 \lambda_{j, m} &: \|j\|_1 = \mu, m\in A_j^{\Omega}\\
 0 &: \text{ otherwise, }
 \end{cases}
\]
we have $\id = \sum_{\mu=0}^{\infty} \id_{\mu}$.
Of course, the identity 
$$
   e_k(\id_\mu:s^{r_0}_{p_0,q_0}a(\Omega) \to s^{r_1}_{p_1,q_1}a^{\dag}(\Omega)) = 
   e_k(\id'_\mu:s^{r_0}_{p_0,q_0}a(\Omega)_\mu \to s^{r_1}_{p_1,q_1}a^{\dag}(\Omega)_{\mu})
$$
holds true, where $\id_{\mu}'$ denotes the corresponding embedding operator on the respective block \eqref{block}. Although these operators have the same mapping properties we use different notations to formally distinguish between them. 
If $a = a^{\dag} = b$ we also have, for a finite index set $I$, that 
\begin{equation}\label{ek:lp(lq)}
  \begin{split}
	&e_k\Big(\sum\limits_{\mu \in I}\id_{\mu} : 
	s^{r_0}_{p_0,q_0}b(\Omega) \to s^{r_1}_{p_1,q_1}b(\Omega)\Big)\\ 
	&~~~~= e_k\Big(
	\id':\ell_{q_0}\Big(\big(s^{r_0}_{p_0,q_0}b(\Omega)_\mu\big)_{\mu \in I}\Big) \to 
	\ell_{q_1}\Big(\big(s^{r_1}_{p_1,q_1}b(\Omega)_{\mu}\big)_{\mu \in I}\Big)\Big)\\
	&~~~~\simeq e_k\Big(
	\id'':\ell_{q_0}\Big((X_\mu)_{\mu \in I}\Big) \to 
	\ell_{q_1}\Big((Y_{\mu})_{\mu \in I}\Big)\Big)\,,
  \end{split}	
\end{equation}
where $X_{\mu} = 2^{\mu(r_0-1/p_0)}\ell_{q_0}^{(\mu+1)^{n-1}}(\ell_{p_0}^{2^\mu})$ and $Y_{\mu} = 2^{\mu(r_1-1/p_1)}\ell_{q_1}^{(\mu+1)^{n-1}}(\ell_{p_1}^{2^\mu})$, which means $d_{\mu} = 2^{\mu}$ and $b_\mu = (\mu+1)^{n-1}$ in the notation of \eqref{lp(lq)}. In particular, we have $b_{\mu} \lesssim d_{\mu}$\,.

\subsection{Entropy numbers}
As a consequence of the boundedness of certain restriction and extension operators, see \cite[4.5]{vybiral:2006:diss},
the investigation of entropy numbers of Besov space embeddings can be shifted to the sequences spaces side. We formulate our first result in the framework of sequence spaces, which improves the upper bound. More specifically, we prove that the lower bound in \eqref{vyb} is sharp in the case that $0 \leq 1/p_0-1/p_1<r_0-r_1 \leq 1/q_0-1/q_1$, which also includes the limiting case $r_0-r_1 = 1/q_0 - 1/q_1$\,. What is known in this direction is summarized in Remark \ref{rem_vyb} below. 

\begin{proposition}\label{thm:besov:res1} Let $\Omega$ be a bounded domain and $0<q_0 < q_1 \leq \infty$, $0<p_0\leq p_1\leq \infty$ such that 
$$
    1/p_0-1/p_1<r_0-r_1 \leq 1/q_0-1/q_1\,.
$$
Then we have  
\begin{equation}\nonumber
      e_m(\id:s^{r_0}_{p_0,q_0}b(\Omega) \to s^{r_1}_{p_1,q_1}b(\Omega)) \simeq m^{-(r_0-r_1)}\quad,\quad m\in \N\,.
\end{equation}
\end{proposition}

\begin{proof} The lower bound follows by \cite[Thm.\ 3.18]{vybiral:2006:diss}. The upper bound is the actual contribution. We argue as follows. 

{\em Step 1.} Put $\varrho:=\min\{1,p_1,q_1\}$ and fix $m\geq m_0$, where $m_0$ is large enough (depending on $p_0,p_1,q_0,q_1,r_0,r_1$). We decompose the identity operator $\id$ as follows
\begin{equation}\label{id_dec}
	\id = \Big(\sum\limits_{\mu = 0}^{L_m} \id_{\mu}\Big) + \Big(\sum\limits_{\mu = L_m+1}^{M_m+L_m} \id_{\mu}\Big) + \Big(\sum\limits_{\mu=M_m+L_m+1}^{\infty} \id_{\mu}\Big)\,,
\end{equation}
where $L_m := \lfloor \log_2(m) \rfloor$ and $M_m := \lfloor m/8 \rfloor$. With an eye on Proposition \ref{lem:ub_by_dyadic_covering} this means in particular that $m \geq 8L_m$ and $m\geq 8M_m$ (for $m$ large enough). Using \eqref{ek(T1+T2)} we obtain
\begin{equation}\label{e2m+1}
	e_{2m}(\id)^{\varrho} \lesssim e_m\Big(\sum\limits_{\mu = 0}^{L_m} \id_{\mu}\Big)^{\varrho} +
	e_m\Big(\sum\limits_{\mu = L_m+1}^{M_m+L_m}\id_{\mu}\Big)^{\varrho} + 
	\sum\limits_{\mu = L_m+M_m+1}^{\infty} \|\id_\mu\|^{\varrho}\,.
\end{equation}

{\em Step 2.} We estimate the first summand. By \eqref{ek:lp(lq)} this breaks down to the entropy numbers
\begin{equation}\label{e_m}
	e_m\Big(
	\id:\ell_{q_0}\Big((X_\mu)_{\mu \in I}\Big) \to 
	\ell_{q_1}\Big((Y_{\mu})_{\mu \in I}\Big)\Big)   
\end{equation}
with $X_\mu, Y_\mu$ chosen as after \eqref{ek:lp(lq)} and $I$ denotes the range for $\mu$. Putting 
$$
p:=q_0, r:=q_1, q := p_0, u:=p_1, d_\mu := 2^\mu, b_{\mu} := (\mu+1)^{n-1}, \alpha := r_0-1/p_0
$$
and $\beta :=r_1-1/p_1$ in \eqref{lp(lq)} we may apply Proposition \ref{prop14}\,. Since $m \geq \max\{8L_m, \max\limits_\mu d_\mu\}$ we may apply Proposition \ref{prop14} to obtain
\begin{equation}\label{res10}
     e_m \lesssim \Big(\frac{1}{m}\Big)^{\alpha-\beta+1/q-1/u} = \Big(\frac{1}{m}\Big)^{r_0-r_1}\,.
\end{equation}
Note that, due to Proposition \ref{prop14}, we only used that $r_0-r_1 \leq 1/q_0-1/q_1$. To estimate the first summand in \eqref{e2m+1} it is not needed that $r_0-r_1>1/p_0-1/p_1$.

{\em Step 3.} Let us care for the second summand in \eqref{e2m+1}. Clearly, it can be reduced to~\eqref{e_m} with spaces $X_\mu, Y_\mu$ defined analogously, but with $\mu$ running this time in the range
$$I = \{L_m+1,...,L_m+M_m\}.$$
Hence, we have $b := \sharp I = M_m \leq \min\limits_\mu d_\mu$\,. We apply Proposition \ref{prop15} to end up with \eqref{res10}. Note, that we have used here only that $\alpha-\beta>0$, or, equivalently, $r_0-r_1>1/p_0-1/p_1$\,.

{\em Step 4.} Finally, we deal with the third summand in \eqref{e2m+1}. Clearly, we have
$$\|\id_{\mu}\| \lesssim 2^{-\mu(r_0-r_1-1/p_0 + 1/p_1)}.$$
This gives 
\begin{equation}
  \begin{split}
    \sum\limits_{\mu = M_m+L_m+1}^{\infty}\|\id_\mu\|^{\varrho} 
	&\lesssim  \sum\limits_{\mu = M_m+L_m+1}^{\infty} 
	2^{-\varrho\mu(r_0-r_1-1/p_0 + 1/p_1)}\\
	&\simeq 2^{-\varrho(m/8+L_m)(r_0-r_1-1/p_0 + 1/p_1)}\\
	&\lesssim m^{-\varrho(r_0-r_1)}\,.
  \end{split}
\end{equation}
This concludes the proof.
\end{proof}
In the next theorem we consider the situation where a Besov type sequence space compactly embeds into a Triebel-Lizorkin type sequence space. This setting is particularly important, since it leads to results with target space $L_p$.

\begin{proposition}\label{thm:besov:res2} Let $\Omega$ be a bounded domain and $0<q_0 < q_1 \leq \infty$, $0<p_0\leq p_1 < \infty$, $q_0 <p_0$, $r_0>r_1$ such that  
$$
    1/p_0-1/p_1<r_0-r_1 \leq 1/q_0-1/q_1\,.
$$
Then we have  
\begin{equation}\nonumber
      e_m(\id:s^{r_0}_{p_0,q_0}b(\Omega) \to s^{r_1}_{p_1,q_1}f(\Omega)) \simeq m^{-(r_0-r_1)}\,.
\end{equation}
\end{proposition}

\begin{proof} Again, the lower bounds follow from \cite[Thm.\ 3.18]{vybiral:2006:diss}.

{\em Step 1.} In case $p_1 > q_1$ we use the commutative diagram in Figure~\ref{fig1}.
\begin{figure}[t]
	\centering
		\begin{tikzpicture}
			\tikzset{node distance=3cm, auto}
	      	\node (H) {$s^{r_0}_{p_0,q_0}b(\Omega)$};
			\node (L) [right of =H] {$s^{r_1}_{p_1,q_1}f(\Omega)$};
			\node (T) [below of =L] {$s^{r_1}_{p_1,q_1}b(\Omega)$};
			\draw[->] (H) to node[left] {$\id$} (T);
			\draw[->] (H) to node {$\id$} (L);
			\draw[->] (T) to node[right] {$\id$} (L);
		\end{tikzpicture}
	\caption{Decomposition of $\id$ in the case $p_1 \geq q_1$.}
	\label{fig1}
\end{figure}
Then we have by~\eqref{ek(RS)} and~\eqref{ek(RS)2}
\begin{equation}
  \begin{split}
	&e_m(\id:s^{r_0}_{p_0,q_0}b(\Omega) \to s^{r_1}_{p_1,q_1}f(\Omega)) \\
	&~~~~\lesssim e_m(\id:s^{r_0}_{p_0,q_0}b(\Omega) \to s^{r_1}_{p_1,q_1}b(\Omega))
	\cdot \|\id:s^{r_1}_{p_1,q_1}b(\Omega) \to s^{r_1}_{p_1,q_1}f(\Omega)\|\,.
  \end{split}
\end{equation}  
Hence, we may use Proposition \ref{thm:besov:res1} and obtain
  \begin{equation}
	e_m(\id:s^{r_0}_{p_0,q_0}b(\Omega) \to s^{r_1}_{p_1,q_1}f(\Omega)) \\
	\lesssim m^{-(r_0-r_1)}\,.
  \end{equation}

{\em Step 2.} Now we consider $p_1 < q_1$. After decomposing the identity operator in an analogous way as in \eqref{id_dec} and \eqref{e2m+1} we use the commutative diagrams in Figure \ref{fig2} for the first and second summand, respectively. 
In fact, for the first summand in \eqref{e2m+1} we obtain by \eqref{ek(RS)2}
\begin{equation}
  \begin{split}
	&e_m\Big(\sum\limits_{\mu = 0}^{L_m} 
	\id_{\mu}:s^{r_0}_{p_0,q_0}b(\Omega) \to 
	s^{r_1}_{p_1,q_1}f(\Omega)\Big)\\
	&~~~~\lesssim e_m\Big(\sum\limits_{\mu = 0}^{L_m} 
	\id_{\mu}:s^{r_0}_{p_0,q_0}b(\Omega) \to 
	s^{r_1}_{q_1,q_1}b(\Omega)\Big)\cdot 
	\|\id:s^{r_1}_{q_1,q_1}b(\Omega) \to s^{r_1}_{p_1,q_1}f(\Omega)\|\,.
   \end{split}	
\end{equation}
Note, that the identity operator is bounded since $\Omega$ is a bounded domain. 
\begin{figure}[h]
	\centering
	\begin{minipage}{0.48\textwidth}
		\centering
		\begin{tikzpicture}
			\tikzset{node distance=3cm, auto}
	      	\node (H) {$s^{r_0}_{p_0,q_0}b(\Omega)$};
			\node (L) [right of =H] {$s^{r_1}_{p_1,q_1}f(\Omega)$};
			\node (T) [below of =L] {$s^{r_1}_{q_1,q_1}b(\Omega)$};
			\draw[->] (H) to node[left] {$\id_I$} (T);
			\draw[->] (H) to node {$\id_I$} (L);
			\draw[->] (T) to node[right] {$\id$} (L);
		\end{tikzpicture}
	\end{minipage}
	\begin{minipage}{0.48\textwidth}
		\centering
		\begin{tikzpicture}						
			\tikzset{node distance=3cm, auto}
			\node (H) {$s^{r_0}_{p_0,q_0}b(\Omega)$};
			\node (L) [right of =H] {$s^{r_1}_{p_1,q_1}f(\Omega)$};
			\node (T) [below of =L] {$s^{r_1}_{p_1,p_1}b(\Omega)$};
			\draw[->] (H) to node[left] {$\id_I$} (T);
			\draw[->] (H) to node {$\id_I$} (L);
			\draw[->] (T) to node[right] {$\id$} (L);
		\end{tikzpicture}
	\end{minipage}
	\caption{Decomposition of the operator $\id_I = \sum_{\mu\in I} \id_\mu$ in the case $p_1 < q_1$\,.}
	\label{fig2}    
\end{figure}
Furthermore, the entropy numbers
$$
   e_m\Big(\sum\limits_{\mu = 0}^{L_m} 
	\id_{\mu}:s^{r_0}_{p_0,q_0}b(\Omega) \to s^{r_1}_{q_1,q_1}b(\Omega)\Big) \lesssim m^{-(r_0-r_1)} 
$$
can be estimated by the same reasoning as in Step 2 of the proof of Proposition~\ref{thm:besov:res1}. Note that $r_0-r_1$ may be smaller than  $1/p_0 - 1/q_1$. However, this is not important for the argument (based on Proposition~\ref{prop14}). It remains to consider the second summand in \eqref{e2m+1}. Here we use the right diagram in Figure~\ref{fig2} and obtain
\begin{equation}
  \begin{split}
	&e_m\Big(\sum\limits_{\mu = L_m+1}^{M_m+L_m} 
	\id_{\mu}:s^{r_0}_{p_0,q_0}b(\Omega) \to 
	s^{r_1}_{p_1,q_1}f(\Omega)\Big)\\
	&~~~~\lesssim e_m\Big(\sum\limits_{\mu = L_m+1}^{M_m+L_m} 
	\id_{\mu}:s^{r_0}_{p_0,q_0}b(\Omega) \to 
	s^{r_1}_{p_1,p_1}b(\Omega)\Big)\cdot 
	\|\id:s^{r_1}_{p_1,p_1}b(\Omega) \to s^{r_1}_{p_1,q_1}f(\Omega)\|\,.
   \end{split}	
\end{equation}
We continue to estimate the appearing entropy numbers as in Step 3 of the proof of Proposition \ref{thm:besov:res1}. Note that $r_0-r_1$ might be larger than $1/q_0-1/p_1$. However, for the argument, we only need $r_0-r_1>1/p_0-1/p_1$. This concludes the proof. 
\end{proof}

Let us finally consider the situation, where a Triebel-Lizorkin type sequence space compactly embeds into a Besov type sequence space. 

\begin{proposition}\label{thm:besov:res3} Let $0<q_0 < q_1 \leq \infty$, $0<p_0\leq p_1 < \infty$, $q_1 > p_1$, $r_0>r_1$ such that  
$$
    1/p_0-1/p_1<r_0-r_1 \leq 1/q_0-1/q_1\,,
$$
and let $\Omega$ be a bounded domain.
Then we have  
\begin{equation}\label{res1}
      e_m(\id:s^{r_0}_{p_0,q_0}f(\Omega) \to s^{r_1}_{p_1,q_1}b(\Omega)) \simeq m^{-(r_0-r_1)}\,.
\end{equation}
\end{proposition}

\begin{proof} The lower bound follows from \cite[Thm.\ 3.18]{vybiral:2006:diss}. 

{\em Step 1.} For the upper bound in the case $p_0<q_0$ we may use the commutative diagram in Figure \ref{fig3} below to decompose the identity operator. Afterwards, we use \eqref{ek(RS)} to reduce everything to the situation in Proposition \ref{thm:besov:res1}.
\begin{figure}[h]
	\centering
		\begin{tikzpicture}
			\tikzset{node distance=3cm, auto}
	      	\node (H) {$s^{r_0}_{p_0,q_0}f(\Omega)$};
			\node (L) [right of =H] {$s^{r_1}_{p_1,q_1}b(\Omega)$};
			\node (T) [below of =L] {$s^{r_0}_{p_0,q_0}b(\Omega)$};
			\draw[->] (H) to node[left] {$\id$} (T);
			\draw[->] (H) to node {$\id$} (L);
			\draw[->] (T) to node[right] {$\id$} (L);
		\end{tikzpicture}
	\caption{Decomposition of $\id$ in the case $p_0 < q_0$.}
	\label{fig3}
\end{figure}

{\em Step 2.} In case $p_0>q_0$ we argue analogously to Step 2 of Proposition \ref{thm:besov:res2}. This time we use the decompositions in Figure \ref{fig4} for the first and second summand in \eqref{e2m+1}, respectively.
\end{proof}

Unfortunately, we were not able to find a corresponding result for the $f-f$ situation. So, this remains an open problem. 

\begin{remark}\label{rem_vyb}
To clarify the contribution of this paper, let us briefly recapitulate the known results and open questions which motivated this work. For several results and historical remarks on the subject we refer to \cite{dung/ullrich/temlyakov:2015:hyperbolic_cross} and the references therein. In particular, Vyb\'iral \cite[Thm.\ 4.9]{vybiral:2006:diss} proved for $0<p_0\leq p_1 \leq \infty$ and $0<q_0\leq q_1 \leq \infty$ in the case of \emph{small smoothness}
$$1/p_0-1/p_1<r \leq 1/q_0-1/q_1,$$
that there is for any $\varepsilon>0$ a number $C_{\varepsilon}>0$ such that 
\begin{align}\label{vyb}
 c m^{-r} \leq e_m(\id: s_{p_0,q_0}^{r_0}b(\Omega) \to s_{p_1,q_1}^{r_1}b(\Omega)) \leq C_\varepsilon m^{-r} (\log m)^\varepsilon, \quad m \geq 2.
\end{align}
The result is a direct consequence of the bound for $r>\max\{1/p_0-1/p_1,1/q_0-1/q_1\}$ (the case of ``large smoothness''), saying that 
\begin{equation}\label{vyblarge}
    e_m(\id: s_{p_0,q_0}^{r_0}b(\Omega) \to s_{p_1,q_1}^{r_1}b(\Omega)) \simeq m^{-r}(\log m)^{(n-1)(r_0-r_1-1/q_0+1/q_1)_+}\,.
\end{equation}
In fact, the entropy numbers in \eqref{vyb} can be bounded from above by $$e_m(\id: s_{p_0,q^{\ast}}^{r_0}b(\Omega) \to s_{p_1,q_1}^{r_1}b(\Omega))$$ if $q^{\ast}\geq q_0$. Now choose 
$q_1>q^{\ast}>q_0$ such that $1/q^{\ast}-1/q_1+\varepsilon/(n-1)=r_0-r_1>1/q^{\ast}-1/q_1$, which, together with \eqref{vyblarge} and $q_0$ replaced by $q^{\ast}$, implies \eqref{vyb}.  \qed
\end{remark}

The propositions proved above allow to improve a number of existing results for the entropy numbers of the embedding
$$
\Id:S^{r_0}_{p_0,q_0}A(\Omega) \to S^{r_1}_{p_1,q_1}A^{\dag}(\Omega)\,.
$$

\begin{theorem}\label{sec:besov:lem:literally} Let $\Omega$ be a bounded domain and $A, A^{\dag} \in \{B,F\}$ but $(A,A^{\dag}) \neq (F,F)$. Let $0<q_0 < q_1 \leq \infty$, $0<p_0\leq p_1 < \infty$ and $r_0>r_1$ such that  
$$
    1/p_0-1/p_1<r_0-r_1 \leq 1/q_0-1/q_1\,.
$$
In addition, we assume $q_0<p_0$ if $(A, A^{\dag}) = (B,F)$ and $q_1>p_1$ if $(A, A^{\dag}) = (F,B)$, respectively. Then it holds 
$$
	e_m(\Id:S^{r_0}_{p_0,q_0}A(\Omega) \to S^{r_1}_{p_1,q_1}A^{\dag}(\Omega))
\simeq m^{-(r_0-r_1)}\quad,\quad m\in \N\,.
$$
\end{theorem}

\begin{proof} The result is a direct consequence of Propositions \ref{thm:besov:res1}, \ref{thm:besov:res2}, \ref{thm:besov:res3} and the machinery described in the proof of \cite[Thm.\ 4.11]{vybiral:2006:diss}. 
\end{proof}

\noindent As a corollary of Theorem \ref{sec:besov:lem:literally}, we obtain the following result, which settles Open Problem~6.4 in \cite{dung/ullrich/temlyakov:2015:hyperbolic_cross}.

\begin{corollary}\label{sec:besov:cor:op1} Let $\Omega$ be as above. Let further  $0<q<p_0\leq p_1$, $1<p_1<\infty$ and 
$1/p_0-1/p_1<r\leq 1/q-1/2$. Then we have
\begin{equation}\nonumber
		e_m(\Id:S^r_{p_0,q}B(\Omega)\to L_{p_1}(\Omega)) \simeq m^{-r}\,.
\end{equation}
\end{corollary}
\begin{proof}
Identifying $S^0_{p_1,2}F(\Omega) = L_{p_1}(\Omega)$ in the case $1<p_1<\infty$, the result is a direct consequence of Theorem \ref{sec:besov:lem:literally}.
\end{proof}

\noindent With the final corollary below (from Theorem \ref{sec:besov:lem:literally}) we close some more gaps in \cite[Thm.\ 4.18 (ii), (iii)]{vybiral:2006:diss}.

\begin{corollary}\label{sec:besov:cor:op2} Let $\Omega$ be as above. 
We have the following sharp bounds for entropy numbers.

\noindent{\em (i)} Let $1 < p \leq \infty$ and $1/p<r\leq 1$. Then, we have 
$$
    e_m(\Id:S^r_{p,1}B(\Omega) \to S^0_{\infty,\infty}B(\Omega))
    \simeq m^{-r}\,.
$$

\noindent{\em (ii)} Let $1 < p < q <\infty$ and $1/p-1/q<r\leq 1/2$. Then, we have 
$$
    e_m(\Id:S^r_pW(\Omega) \to S^0_{q,\infty}B(\Omega)) \simeq m^{-r}\,.
$$

\noindent{\em{(iii)}} Let $0< q<p \leq \infty$, $q<1$ and 
$1/p < r \leq 1/q - 1$. Then, we have
\[ 
 e_m(\Id:S^r_{p,q}B(\Omega) \to L_\infty(\Omega)) \simeq m^{-r}.
\]
\end{corollary}

\begin{remark} Entropy numbers of mixed smoothness Sobolev-Besov embeddings into $L_p$, where~$1\leq p\leq \infty$, recently gained significant interest, see  \cite{Te17} and \cite{Ro16_1,Ro16_2,Ro19}. There are some fundamental open problems connected with $p=\infty$, see \cite[2.6, 6.4, 6.5]{dung/ullrich/temlyakov:2015:hyperbolic_cross}. Interestingly, when choosing the third index $q$ small enough in Corollaries \ref{sec:besov:cor:op1}, \ref{sec:besov:cor:op2} we get rid of the logarithm. 
\end{remark}

\begin{figure}[h]
	\centering
	\begin{minipage}{0.48\textwidth}
		\centering
		\begin{tikzpicture}
			\tikzset{node distance=3cm, auto}
	      	\node (H) {$s^{r_0}_{p_0,q_0}f(\Omega)$};
			\node (L) [right of =H] {$s^{r_1}_{p_1,q_1}b(\Omega)$};
			\node (T) [below of =L] {$s^{r_0}_{q_0,q_0}b(\Omega)$};
			\draw[->] (H) to node[left] {$\id$} (T);
			\draw[->] (H) to node {$\id_I$} (L);
			\draw[->] (T) to node[right] {$\id_I$} (L);
		\end{tikzpicture}
	\end{minipage}
	\begin{minipage}{0.48\textwidth}
		\centering
		\begin{tikzpicture}						
			\tikzset{node distance=3cm, auto}
			\node (H) {$s^{r_0}_{p_0,q_0}f(\Omega)$};
			\node (L) [right of =H] {$s^{r_1}_{p_1,q_1}b(\Omega)$};
			\node (T) [below of =L] {$s^{r_0}_{p_0,p_0}b(\Omega)$};
			\draw[->] (H) to node[left] {$\id$} (T);
			\draw[->] (H) to node {$\id_I$} (L);
			\draw[->] (T) to node[right] {$\id_I$} (L);
		\end{tikzpicture}
	\end{minipage}
	\caption{Decomposition of the operator $\id_I = \sum_{\mu\in I} \id_\mu$ in the case $p_0 > q_0$\,.}
	\label{fig4}    
\end{figure}

\paragraph{Acknowledgment.} The authors would like to thank Dinh D\~ung, Thomas K\"uhn, Van Kien Nguyen, Winfried Sickel and Vladimir N.\ Temlyakov for several discussions on the topic. They would also like to thank two anonymous referees for their valuable comments. 
T.U. and S.M.\ would like to acknowledge support by the DFG Ul-403/2-1 and by the Fraunhofer Cluster of Excellence ``Cognitive Internet Technologies''.

\bibliographystyle{abbrv}

\begin{thebibliography}{10}
\bibitem{akashi:1986:entropy}
S.~Akashi.
\newblock An operator theoretical characterization of {$\varepsilon$}-entropy
  in {G}aussian processes.
\newblock {\em Kodai Mathem. Journ.}, 9(1):58--67, 1986.

\bibitem{aoki:1942}
T.~Aoki.
\newblock Locally bounded linear topological spaces.
\newblock {\em Proceedings of the Imperial Academy}, 18(10):588--594, 1942.

\bibitem{billingsley:1965:ergodic_theory}
P.~Billingsley.
\newblock {\em Ergodic theory and information}.
\newblock John Wiley \& Sons, Inc., New York-London-Sydney, 1965.

\bibitem{carl:1981}
B.~Carl.
\newblock Entropy numbers, {$s$}-numbers, and eigenvalue problems.
\newblock {\em J. Funct. Anal.}, 41(3):290--306, 1981.

\bibitem{carl/stefani:1990}
B.~Carl and I.~Stephani.
\newblock {\em Entropy, compactness and the approximation of operators}.
\newblock Cambrige Univ. Press, Cambridge, 1990.

\bibitem{dirksen/ullrich:2017:gelfand}
S.~{Dirksen} and T.~{Ullrich}.
\newblock {Gelfand numbers related to structured sparsity and Besov space
  embeddings with small mixed smoothness}.
\newblock {\em J. Complexity}, 48:69--102, 2018.

\bibitem{dung:2001:nonlinear}
D.~D{\~u}ng.
\newblock Non-linear approximations using sets of finite cardinality or finite
  pseudo-dimension.
\newblock volume~17, pages 467--492. 2001.
\newblock 3rd Conference of the Foundations of Computational Mathematics
  (Oxford, 1999).

\bibitem{dung/ullrich/temlyakov:2015:hyperbolic_cross}
D.~{D}{\~u}ng, V.~N. {T}emlyakov, and T.~{U}llrich.
\newblock {\em {H}yperbolic {C}ross {A}pproximation}.
\newblock Advanced Courses in Mathematics. CRM Barcelona.
  Birkh\"auser/Springer, 2018.

\bibitem{edmunds/netrusov:1998:entropy}
D.~E. Edmunds and Y.~Netrusov.
\newblock Entropy numbers of embeddings of {S}obolev spaces in {Z}ygmund
  spaces.
\newblock {\em Studia Math.}, 128(1):71--102, 1998.

\bibitem{edmunds/netrusov:2014:schuett}
D.~E. Edmunds and Y.~Netrusov.
\newblock Sch\"utt's theorem for vector-valued sequence spaces.
\newblock {\em J. Approx. Theory}, 178:13--21, 2014.

\bibitem{edmunds/triebel:1996:function_spaces}
D.~E. Edmunds and H.~Triebel.
\newblock {\em Function spaces, entropy numbers, differential operators},
  volume 120 of {\em Cambridge Tracts in Mathematics}.
\newblock Cambridge University Press, Cambridge, 1996.

\bibitem{foucart/pajor/rauhut/ullrich:2010:gelfand}
S.~Foucart, A.~Pajor, H.~Rauhut, and T.~Ullrich.
\newblock The {G}elfand widths of {$\ell_p$}-balls for {$0< p\leq 1$}.
\newblock {\em Journal of Complexity}, 26(6):629--640, 2010.

\bibitem{rauhut/foucart:compressive_sensing}
S.~Foucart and H.~Rauhut.
\newblock {\em A mathematical introduction to compressive sensing}.
\newblock Applied and Numerical Harmonic Analysis. Birkh\"auser/Springer, New
  York, 2013.

\bibitem{gilbert:1952}
E.~{G}ilbert.
\newblock A comparison of signalling alphabets.
\newblock {\em Bell System Tech. Jour.}, 31:504--522, 1952.

\bibitem{guedon/litvak:2000:euclidean_projections}
O.~Gu\'edon and A.~E. Litvak.
\newblock Euclidean projections of a {$p$}-convex body.
\newblock In {\em Geometric aspects of functional analysis}, volume 1745 of
  {\em Lecture Notes in Math.}, pages 95--108. Springer, Berlin, 2000.

\bibitem{hinrichs/kolleck/vybiral:2016:carl_quasi}
A.~Hinrichs, A.~Kolleck, and J.~Vyb{\'i}ral.
\newblock Carl's inequality for quasi-{B}anach spaces.
\newblock {\em J. Funct. Anal.}, 271(8):2293--2307, 2016.

\bibitem{KeVy17}
H.~Kempka and J.~Vyb\'iral.
\newblock Volumes of unit balls of mixed sequence spaces.
\newblock {\em Mathematische Nachrichten}, 290(8-9):1317--1327, 2017.

\bibitem{kolmogorov:1956:entropy}
A.~N. Kolmogorov.
\newblock On certain asymptotic characteristics of completely bounded metric
  spaces.
\newblock {\em Dokl. Akad. Nauk SSSR (N.S.)}, 108:385--388, 1956.

\bibitem{kolmogorov:1956:entropy_dynamical_system}
A.~N. Kolmogorov.
\newblock A new metric invariant of transient dynamical systems and
  automorphisms in {L}ebesgue spaces.
\newblock {\em Dokl. Akad. Nauk SSSR (N.S.)}, 119:861--864, 1958.

\bibitem{koenig:1986:eigenvalue_distributions}
H.~K\"onig.
\newblock {\em Eigenvalue distribution of compact operators}.
\newblock Birkh\"auser Verlag, Basel, 1986.

\bibitem{kuelbs:1993:metric_entropy}
J.~Kuelbs and W.~V. Li.
\newblock Metric entropy and the small ball problem for {G}aussian measures.
\newblock {\em Journal of Functional Analysis}, 116(1):133--157, 1993.

\bibitem{kuehn:2001:lower_estimate}
T.~K\"uhn.
\newblock A lower estimate for entropy numbers.
\newblock {\em J. Approx. Theory}, 110(1):120--124, 2001.

\bibitem{kuehn/leopold/sickel/skrzypczak:2006:entropy}
T.~K{\"u}hn, H.-G. Leopold, W.~Sickel, and L.~Skrzypczak.
\newblock Entropy numbers of embeddings of weighted {B}esov spaces.
\newblock {\em Constr. Approx.}, 23(1):61--77, 2006.

\bibitem{li:1999:approximation_metric_entropy}
W.~V. Li, W.~Linde, et~al.
\newblock Approximation, metric entropy and small ball estimates for {G}aussian
  measures.
\newblock {\em The Annals of Probability}, 27(3):1556--1578, 1999.

\bibitem{lorentz/golitschek/makovoz:1996:cs_advanced}
G.~G. Lorentz, M.~v. Golitschek, and Y.~Makovoz.
\newblock {\em Constructive approximation}, volume 304 of {\em Grundlehren der
  Mathematischen Wissenschaften [Fundamental Principles of Mathematical
  Sciences]}.
\newblock Springer-Verlag, Berlin, 1996.
\newblock Advanced problems.

\bibitem{Ma75}
V.~E. Ma\u{\i}orov.
\newblock Discretization of the problem of diameters.
\newblock {\em Uspehi Mat. Nauk}, 30(6(186)):179--180, 1975.

\bibitem{mayer:2018:thesis}
S.~Mayer.
\newblock {\em Preasymptotics via metric entropy}.
\newblock Dr. Hut Verlag, Munich, 2018.
\newblock PhD thesis.

\bibitem{pietsch:1978:operator}
A.~Pietsch.
\newblock {\em Operator ideals}, volume~16.
\newblock Deutscher Verlag der Wissenschaften, 1978.

\bibitem{pietsch:s-numbers_and_eigenvalues}
A.~Pietsch.
\newblock {\em Eigenvalues and {$s$}-numbers}, volume~13 of {\em Cambridge
  Studies in Advanced Mathematics}.
\newblock Cambridge University Press, Cambridge, 1987.

\bibitem{pinkus:nwidth}
A.~Pinkus.
\newblock {\em {$n$}-widths in approximation theory}, volume~7 of {\em
  Ergebnisse der Mathematik und ihrer Grenzgebiete (3) [Results in Mathematics
  and Related Areas (3)]}.
\newblock Springer-Verlag, Berlin, 1985.

\bibitem{pontrjagin/schnirelmann:1932:entropy}
L.~Pontrjagin and L.~Schnirelmann.
\newblock Sur une propri\'et\'e m\'etrique de la dimension.
\newblock {\em Ann. of Math. (2)}, 33(1):156--162, 1932.

\bibitem{rolewicz:1957}
S.~Rolewicz.
\newblock On a certain class of linear metric spaces.
\newblock {\em Bull. Acad. Polon. Sci}, 5:471--473, 1957.

\bibitem{Ro16_2}
A.~S. Romanyuk.
\newblock Entropy numbers and widths of the classes {$B^r_{p,\theta}$} of
  periodic functions of many variables.
\newblock {\em Ukra\"{i}n. Mat. Zh.}, 68(10):1403--1417, 2016.

\bibitem{Ro16_1}
A.~S. Romanyuk.
\newblock Estimation of the entropy numbers and {K}olmogorov widths for the
  {N}ikol'skii-{B}esov classes of periodic functions of many variables.
\newblock {\em Ukrainian Math. J.}, 67(11):1739--1757, 2016.
\newblock Translation of Ukra\"{i}n. Mat. Zh. {{\bf{6}}7} (2015), no. 11,
  1540--1556.

\bibitem{Ro19}
A.~S. Romanyuk.
\newblock Entropy {N}umbers and {W}idths for the {N}ikol'skii--{B}esov
  {C}lasses of {F}unctions of {M}any {V}ariables in the {S}pace {$L_\infty$}.
\newblock {\em Anal. Math.}, 45(1):133--151, 2019.

\bibitem{schmeisser/triebel:1987:fourier_analysis}
H.-J. Schmeisser and H.~Triebel.
\newblock {\em Topics in {F}ourier analysis and function spaces}.
\newblock A Wiley-Interscience Publication. John Wiley \& Sons, Ltd.,
  Chichester, 1987.

\bibitem{schuett:1984:entropy}
C.~Sch\"utt.
\newblock Entropy numbers of diagonal operators between symmetric {B}anach
  spaces.
\newblock {\em J. Approx. Theory}, 40(2):121--128, 1984.

\bibitem{temlyakov:2011:greedy}
V.~Temlyakov.
\newblock {\em Greedy approximation}, volume~20 of {\em Cambridge Monographs on
  Applied and Computational Mathematics}.
\newblock Cambridge University Press, Cambridge, 2011.

\bibitem{Te17}
V.~Temlyakov.
\newblock On the entropy numbers of the mixed smoothness function classes.
\newblock {\em J. Approx. Theory}, 217:26--56, 2017.

\bibitem{varshamov:1957}
R.~{V}arshamov.
\newblock Estimate of the number of signals in error correcting codes.
\newblock {\em Dokl. Acad. Nauk SSSR}, 117:739--741, 1952.

\bibitem{vybiral:2006:diss}
J.~Vyb{\'i}ral.
\newblock Function spaces with dominating mixed smoothness.
\newblock {\em Dissertationes Math. (Rozprawy Mat.)}, 436:73, 2006.

\bibitem{williamson/smola/schoelkopf:1999:entropy}
R.~C. Williamson, A.~J. Smola, and B.~Sch{\"o}lkopf.
\newblock Entropy numbers, operators and support vector kernels.
\newblock In P.~Fischer and H.~U. Simon, editors, {\em Computational Learning
  Theory}, pages 285--299, Berlin, Heidelberg, 1999. Springer Berlin
  Heidelberg.

\end{thebibliography}

\end{document}